\documentclass{article}[8pt]

\usepackage[left=1in,top=1in,right=1in,bottom=1in]{geometry}
\usepackage{amsmath,amssymb,amsthm,amsfonts}
\usepackage{times}
\usepackage{enumerate}
\usepackage{epsfig}
\usepackage{graphicx}
\usepackage{pgf}
\usepackage{chemarr}
\usepackage{float}
\restylefloat{table}
\usepackage{relsize}
\usepackage{appendix}
\usepackage[linesnumbered,ruled,commentsnumbered,noline]{algorithm2e}

\usepackage{multirow}
\usepackage[left]{lineno}
\usepackage{hhline}
\usepackage{lipsum}
\usepackage{subcaption}

\newtheorem{theorem}{Theorem}[section]
\newtheorem{lemma}[theorem]{Lemma}

\newtheorem{Proposition}[theorem]{Proposition}

\def \tilde{\widetilde}

\newcommand{\hs}{\hspace}

\def \triple\mid {\mid \! \mid  \! \mid }

\def\<{\langle}
\def\>{\rangle}
\def\~{\tilde}

\date{}

\title{\LARGE \bf
Hybrid Master Equation for Jump-Diffusion Approximation of Biomolecular Reaction Networks }
\author{Derya Alt{\i}ntan $^{1,}$\thanks{This
	author acknowledges support from the Scientific and Technological Research Council of Turkey
	(T{\"{U}}B{\.{I}}TAK), Program no: 3501 Grant no. 115E252},
Heinz Koeppl$^{2,}$\thanks{corresponding author}}

\begin{document}
	
	\maketitle
	\footnotetext[1]{ Department of Mathematics, Sel\c{c}uk University, altintan@selcuk.edu.tr}
	\footnotetext[2]{Department of Electrical Engineering and Information Technology, Technische Universit\"{a}t Darmstadt,\\ \hs*{0.5cm} heinz.koeppl@bcs.tu-darmstadt.de}
	
	\begin{abstract}
		Cellular reactions have  multi-scale nature in the sense that the abundance of molecular species and the magnitude of reaction rates can vary in a wide range. This diversity leads to hybrid models that combine deterministic and stochastic modeling approaches. To reveal this multi-scale nature, we proposed jump-diffusion approximation in a previous study. The key idea behind the model was to partition reactions into fast and slow groups, and then to combine Markov chain updating scheme for the slow set with diffusion (Langevin)
		approach updating scheme for the fast set. Then, the state vector of the model was defined as the summation of the random time change model  and the solution of the Langevin equation. In this study, we have proved 
		that the joint probability density function of the jump-diffusion approximation over the reaction counting process satisfies the hybrid master equation, which is the summation of the chemical master equation and the Fokker-Planck equation. To solve the hybrid master equation, we propose an algorithm using the moments of reaction counters of fast reactions given the reaction counters of slow reactions. Then, we solve a constrained optimization problem for each conditional probability density at the time point of interest utilizing the maximum entropy approach. 
		Based on the multiplication rule for joint probability density functions, we construct the solution of the hybrid master equation. To show the efficiency of the method, we implement it to a canonical model of gene regulation.

		\noindent
		{\bf Keywords:} jump-diffusion approximation, chemical master equation, Fokker-Planck equation, maximum entropy approach
	\end{abstract}

\section{Introduction}
Reaction networks in systems of biology have discrete and stochastic nature \cite{ee:10,ff:02,fcs:10}.
 Ignoring the randomness of the stochastic fluctuations and the discreteness of the number of molecules of species result in  inappropriate models which cannot correctly describe the dynamics of the whole cell. Stochastic modeling approach explains the dynamics of these systems using discrete-state continuous-time Markov chains and 
 describes the state of the system by integer-valued number of molecules of species. In this approach, the state vector of the system satisfies the random time change model (RTCM), which defines the reaction counting processes using Poisson processes \cite{ak:11}. Also, the probability mass function of these systems satisfies a set of  differential equations referred to as the chemical master equation (CME) in the literature \cite{gill:76}. When the number of molecules of the species in the system of interest is very high, the state vector of the system can be defined by real-valued concentrations instead of integer-valued particle numbers. Dynamics of such systems can be  modeled through diffusion approximation, and the state vector of the system satisfies an It\^{o} stochastic differential equation (SDE) known as the chemical Langevin equation (CLE). Similarly, probability density function of these systems suffices the  Fokker-Planck equation (FPE) \cite{gill:00,gill:02}. In the thermodynamic limit, in which the number of molecules of species and the system volume both approach to infinity while the concentrations of species stay constant, the state of the system is given by the reaction rate equation (RRE) of the traditional deterministic modeling approach.

	Cellular reaction systems involve reactions with very different rates and species with very different abundances. Models only based on the traditional deterministic modeling approach fail to account for this nature. Therefore, different hybrid methods that couple the stochastic and deterministic modeling approaches are needed. In general, hybrid methods separate reactions and/or  species into different groups of reactions and/or species, and they use the diffusion or the deterministic modeling approach to  describe the dynamics of fast reactions and/or species with high copy numbers, while Markov chain representation is utilized for slow reactions and/or species with low copy numbers   \cite{ckl:16,ce:18,cdr:09,dez:16,ehl:17,jah:11}. 
	
	A major challenge of modeling the reaction networks using the CME is the curse of dimensionality. Each state of the system under consideration adds one dimension to the corresponding CME. Therefore, when the number of reachable states is very high, it is very difficult to obtain the numerical solution  of the CME. To avoid this drawback, different simulation algorithms, such as Gillespie's stochastic simulation algorithms (SSAs) and their versions, were proposed to obtain the trajectories of the biochemical system of interest \cite{gb:00,gill:07}. The computational cost of these algorithms increases with the size of the model; therefore, it is not appropriate to use them  for  very complicated systems involving many reactions and reactants. 
	
	Moment approximations that analyze the dynamics of  the reaction network under consideration  using  moments of the probability distribution satisfying the corresponding CME are considered as an alternative. 
	In \cite{eng:06}, the author proposed the method of moments that computes the moments for any reaction network from the corresponding CME. In \cite{lkk:09}, a moment closure approximation that obtains finite dimensional ordinary differential equation (ODE) system for the mean and the central moments by truncating the moment equations at a certain order and using the Taylor series is introduced. Another moment closure method that approximates the moments with higher order, compared to the order of truncation, utilizing nonlinear functions of the lower order moments is introduced in  \cite{sh:11}. 
	
	In \cite{hwkt:13}, the authors introduced the method of conditional moments (MCM) that can be considered as the combination of a hybrid method and a moment approximation method. The MCM separates  species into two different classes involving species with high copy number of molecules and species with low copy number of molecules. Based on this decomposition, the joint probability density  function satisfying the corresponding CME is also represented as a product of the marginal probabilities of species with low copy number of molecules and the conditional probabilities of species with high copy number of molecules conditioned on the remaining species with low copy numbers of molecules. To describe the dynamics of species with low copy number of molecules, the authors used  marginal probabilities, while the conditional means and the centered conditional moments are used to model the dynamics of species with high copy numbers. In comparison to \cite{hwkt:13}, in \cite{amw:15}, the authors obtained  moments of the system of interest directly from the corresponding CME without using any partitioning of the species, and the maximum entropy approach is used to construct the corresponding probability distribution. 
	
	In \cite{gak:015}, we developed a jump-diffusion approximation to model multi-scale behavior of cellular reactions. Based on an error bound, we separated reactions into fast and slow groups. We employed diffusion approximation for the fast reactions, while Markov jump process was kept for the slow ones. As a result, the state of the system was defined as the summation of the RTCM and the solution of the corresponding CLE. 
	In this paper, based on this representation, we present the hybrid master equation (HME), which is the evolution equation for the joint probability density  function  of the jump diffusion approximation over the reaction counting process. We prove that the HME is the summation of the corresponding CME and the corresponding FPE \cite{paw:67}. To solve the HME, we obtain  the evolution equation for the marginal probability of slow reactions
and the evolution equations for the conditional moments of the fast reactions given slow  reactions \cite{hwkt:13}. Using the maximum entropy approach, we construct the corresponding conditional probability at the time point of interest, which in turn gives the approximate solution of the corresponding HME.
	
	The rest of the paper is organized as follows: In Section  \ref{basics}, we describe the basic concepts of the stochastic modeling approach. In Section \ref{jump_diffusion}, we give a brief summary of the jump-diffusion approximation. We introduce the HME  in Section \ref{hybrid_master0}. 
	In Section \ref{solution_hme}, we construct an ODE system that will be used to obtain the approximate solution of the HME. In Section \ref{entropy}, we introduce the maximum entropy approach.  In Section \ref{appl_sec}, we present numerical results and also explain the details of how we use the maximum entropy approach to construct the joint probability density function describing the HME. Section \ref{conclusion0} concludes the paper. 
	\newline
	\textbf{Notation}
	\newline
	Before we give the details of the mathematical derivations, we present the basic notations used through the present paper. We represent all random variables and their realizations by upper-case (i.e. \(A\)) and lower case (i.e. \(a\)) symbols, respectively, and we use bold symbols to represent the support of a random variable (i.e. \( \mathbf{A}\)). Also, \(e_j\), \(\bar{e}_j\) denote \((R-L) \times 1\), 
	\(L \times 1\), unit vectors with 1 in the \(j-\)th component and \(0\) in other coordinates. 
	
	\section{Stochastic Modeling of Chemical Kinetics}
	\label{basics}
	In this study, we consider a well-mixed reaction system of  \(M\) species, \(S_1,S_2,\ldots, S_M\), interacting  through \(R \geq 1\) reaction channels \(R_1,R_2,\ldots,R_R\) inside the reaction compartment with volume \(V\). The \(k-\)th reaction channel of the system is described as follows: 
	
	\begin{equation*} 
	\label{eq:system}
	r_{1k}S_1+r_{2k}S_2+r_{3k}S_{3}+\ldots+r_{Mk}S_{M} \stackrel{ \displaystyle \ell_{k}}{\longrightarrow} p_{1k}S_1+p_{2k}S_2+p_{3k}S_{3}+\ldots+p_{Mk}S_{M}, 
	\end{equation*}
	where \(r_{jk}, p_{jk} \in \mathbb{N}\), \(j=1,2,\ldots,M\),  represent the number of molecules of species \(S_j\) consumed and produced with a single occurrence of the reaction \(R_k\), respectively, and \(\ell_k\) is the real-valued stochastic reaction rate constant. 
	Let \(X_i(t) \in \mathbb{N}_{0}\) denote the number of molecules of species \(S_i\), \(i=1,2,\ldots,M\), at time \(t \geq 0\). Then, the state of the system at time \(t\) is  \(X(t)=(X_1(t),X_2(t),\ldots, X_M(t))^{T} \in \mathbb{N}_{0}^{M} \). 
	
	The classical stochastic modeling of biochemical networks assumes that the process of \(X\) is a  continuous time Markov chain (CTMC). In this approach, the state vector, \(X(t)\), is defined as a random variable of the Markov jump process. Each reaction channel \(R_k\) , \(k=1,2,\ldots,R,\) is specified by its stoichiometric   vector (state-change vector) and its propensity function. The stoichiometric   vector \(\nu_k=(\nu_{1k},\nu_{2k},\ldots,\nu_{Mk}) \in \mathbb{Z}^{M}\) with \(\nu_{jk}= p_{jk}-r_{jk}\), \(j=1,2,\ldots,M\),  represents the change in the state of the system after one occurrence of the reaction \(R_k\). In other words, when the reaction \(R_k\) fires, the system state \(X(t)=x\) jumps to a new state \(x+\nu_k\). Given \(X(t)=x\), the probability that one \(R_k\) reaction takes place in the time interval \([t,t+h)\) is  \(a_k(x)h+o(h)\) where \(a_k(x): \mathbb{N}_{0}^{M} \rightarrow \mathbb{R}_{+}\) represents the propensity function calculated by the law of mass action kinetics, i.e., \(a_k(x)=\ell_k \displaystyle \prod_{i=1}^{M}   {x_i \choose r_{ik}}\). Let \(Z_k(t)\) denote the number of occurrence of the reaction \(R_k\)  by the time \(t\), then the state of the system  at time \(t\) can be obtained as follows: \[X(t)=X(0)+\displaystyle \sum_{k=1}^{R}Z_k(t) \nu_k. \]
	If we represent the counting process \(Z_k(t)\) in terms of the independent Poisson process denoted by \(\xi_k\), such that \(Z_k(t)=\xi_k\Big( \displaystyle \int_{0}^{t} a_k(X(s))ds\Big) \), then the state vector of the above CTMC satisfies the following RTCM \cite{ak:11}
	\begin{equation} \label{eq:1}
	X(t)=X(0)+\displaystyle \sum_{k=1}^{R} \xi_{k}\Big( \displaystyle \int_{0}^{t} a_{k}(X(s)) ds \Big) \nu_{k}.
	\end{equation}
	Let define the following probability mass function \[p_t(x) =  \mathrm{P} (X(t)=x).\]
	
	Another way of analyzing this CTMC process is to consider the time evolution of the probability function \(p_t(x)\). This probability  mass function is the solution of the following Kolmogorov's forward equation, which is known as the CME  \cite{gill:92}
	\begin{equation}  
	\label{eq:2}
	\displaystyle \frac{\partial p_t (x) }{\partial t}=\displaystyle \sum_{k=1}^{R} [a_k(x-\nu_{k}) p_t(x-\nu_{k})- a_{k}(x) p_t(x)]. 
	\end{equation}  
	When the number of molecules in the system is very high, then the abundance of the species at time \(t\) can be represented by the real valued concentrations of the form \(U(t)=V^{-1} X(t) \in \mathbb{R}^{M}_{\geq 0}\). In most cases, reaction channels in biochemical systems are bimolecular or monomolecular. If the \(k-\)th reaction channel \(R_k\) is bimolecular or monomolecular, then its propensity function satisfies the equality \(a_{k}(x)= V \tilde{a}_k(u)\)  where \(\tilde{a}_k\) is the propensity function obtained using the deterministic reaction rate \(\tilde{\ell}_k\) \cite{wil:06}. 
	
	It is well known that the centered version of each Poisson process, \(\xi_k\), in  Equation (\ref{eq:1}) can be approximated through the independent Brownian motions \(W_k(t)\) \cite{ak:11,kur:78}. Considering the fact that  \((\xi_{k}(V t)-V t)/ \sqrt{V} \) converges in distribution to the Brownian motion  \(W_k(t)\) for large \(V\), we obtain the diffusion approximation of Equation (\ref{eq:1}) as given below: 
	\begin{equation} \label{eq:3}
	U(t)=U(0)+\displaystyle \sum_{k=1}^{R} \nu_{k} \displaystyle\int_{0}^{t} \tilde{a}_{k}(U(s)) ds+ \frac{1}{\sqrt{V}} \displaystyle \sum_{k=1}^{R} \nu_{k} W_{k} (\displaystyle\int_{0}^{t} \tilde{a}_{k}(U(s)) ds).
	\end{equation}
	
	The first and  the second summand in the right hand-side of Equation (\ref{eq:3}) are called  \textit{drift}  and  \textit{diffusion} terms, respectively. The time derivative of the state vector 
	\(U(t)\) satisfies an SDE, namely the CLE. 
	
	Let define the following probability density function 
	\[q_t(u) du =  \mathrm{P}( U(t) \in [u,u+du]).\] Then, analog of the CME for this continuous process is represented by the following FPE \cite{gill:00,gill:02}
	\begin{equation*}
	\frac{\partial q_t(u)}{\partial t}=-\displaystyle \sum_{i=1}^{M} \frac{\partial}{\partial u_{i}}[(\displaystyle \sum_{k=1}^{R} \nu_{ik}\tilde{a}_{k}(u)) q_t(u)] +\frac{1}{2} \displaystyle \sum_{i,i'=1}^{M}  \frac{\partial^2}{\partial u_{i}  \partial  u_{i'}} \Big[(\displaystyle \sum_{k=1}^{R} \nu_{ik}\nu_{i'k} \tilde{a}_k(u)) q_t(u)\Big]. 
	\end{equation*}
	Cellular processes consist of bimolecular reactions of very different speeds involving reactants of largely different abundances. Therefore, the models based only on  the RTCM or only the diffusion approximation may be inappropriate to dynamics of such multi-scale processes. In \cite{gak:015}, we developed a jump-diffusion approximation to model such processes. 
	In the following section, we will give a summary of this approximation. 
	
	\section{Jump Diffusion Approximation}
	\label{jump_diffusion}
	In  jump-diffusion approximation \cite{gak:015}, we partition the reactions into the fast subgroup, \(\mathcal{C}\), and the slow subgroup, \(\mathcal{D}\), and model the fast group using a diffusion process, while Markov chain representation is kept for the slow group. 
	
	In this approach instead of the CTMC process represented by \(X\), we focus on the scaled abundances 
	$\bar{X}^N_i = X_i/N^{\zeta_i}$, \(i=1,2,\ldots,M\), and the scaled stochastic reaction rates  \(\kappa_j=\ell_j/N^{\eta_j}\), \(j=1,2,\ldots, R\), such that $\bar{X}^N_i = O\left(1\right)$, $\kappa_j = O\left(1\right)$. 
	Naturally, these  scaled quantities will produce new scaled propensity functions as follows
	\[a_k\left(X(t)\right) = N^{\eta_{k}+r_{k}\cdot \zeta}\bar{a}_{k}(\bar{X}^N (t)),\]
	where \(r_k=(r_{1k},r_{2k},\ldots,r_{Mk} )\) and \(\zeta=(\zeta_1,\zeta_2,\ldots,\zeta_M)\). It must be noted that \(\bar{a}_{k}(.)\) functions are also \(O(1)\). Finally, scaling the time  \(t \rightarrow t N^{\theta}\) and defining \(X^{N}(t)=\bar{X}^N(tN^{\theta})\), we transform the state vector, \(X(t)\), given by  Equation (\ref{eq:1})  into the following scaled state vector 
	\begin{align}
	\label{eq:5}
	X^N\left(t\right) &= X^N\left(0\right) + \sum_{k=1}^{R} \xi_k \left( \displaystyle \int_0^t \alpha_k(X^N(s)) ds \right) \ \mu_k,
	\end{align}
	where  \( \alpha_k(X^{N}(t))= N^{\rho_k} \bar{a}_{k} (X^{N}(t))\), \(\rho_k=\theta+\eta_k+r_k \cdot \zeta\) and \(\mu_{ik}=\nu_{ik}/N^{\zeta_i}\), \(i=1,2,\ldots, M\).
	
	Modeling the fast reactions through diffusion approximation and modeling the slow reactions through Markov chains give the state vector of the jump diffusion approximation as follows: 
	\begin{equation}
	Y(t)=Y(0)+ \displaystyle \sum_{i \in \mathcal{D}} \xi_{i}(\displaystyle \int_{0}^{t} \alpha_{i} (Y (s)) ds) \mu_{i} \label{eq:6} 
	+\displaystyle \sum_{j \in \mathcal{C}} \displaystyle \int_{0}^{t}  \alpha_{j}(Y(s)) ds \, \mu_{j}
	+\displaystyle \sum_{j \in \mathcal{C}} W_{j} (\displaystyle \int_{0}^{t} \alpha_{j} (Y(s)) ds) \mu_{j},
	\end{equation}
	where \(Y(0)=X^{N}(0)\), and \(W_j\) is a standard Brownian motion. If \(\tau_1\), \(\tau_2\) denote the successive firing times of reactions from the slow group, then for \(\tau_1< t <\tau_2\), only reactions from the fast group can fire. Therefore, in this time interval, the state vector of the system is given by 
	\begin{equation}
	\label{diffusion}
	Y(t)=Y(\tau_1)+\displaystyle \sum_{j \in \mathcal{C}} \displaystyle \int_{\tau_1}^{t}  \alpha_{j}(Y(s)) ds\, \mu_{j}
	+\displaystyle \sum_{j \in \mathcal{C}} W_{j} (\displaystyle  \int_{\tau_1}^{t}  \alpha_{j} (Y(s)) ds) \mu_{j}.\\
	\end{equation}
	The main contribution of this study is the derivation of an error bound for the mean \(e(t)= \mathsf{E}\mid  X^{N}(t)-Y(t) \mid \), which is used to partition the reaction set into fast and slow subgroups. Based on this error bound, we construct a dynamic partitioning algorithm that takes into account the fact that a fast reaction can return to a slow reaction or vice versa during the course of time.

	By describing the state vector of the system as the summation of purely discrete and purely continuous components, we can  introduce the HME, which defines the joint probability density function of the jump diffusion approximation over the reaction counting process. In the following section, we will obtain the HME. 
	
	\section{Hybrid Master Equation}
	\label{hybrid_master0}
	In jump diffusion approximation, we partition the reaction set into two subsets. As mentioned before, the first subset  \(\mathcal{C}\) involves reactions modeled by diffusion approximation, while the rest of the reactions constituting the slow set  \(\mathcal{D}\) are modeled by Markov chains. In the rest of the study, we will consider 
	that there are \(L\) slow reactions, i.e., \(\mid  \mathcal{D}\mid=L\), and \(R-L\) fast reactions in the system, i.e., \(\mid  \mathcal{C}\mid=R-L\).

	Let \(Z(t)=(Z_1(t),Z_2(t),\ldots,Z_R(t))^{T}\) be a vector of reaction counters such that \(Z_{i}(t)\) denotes the number of occurrences of  
	the reaction \(R_i\), \(i=1,2,\ldots,R\), during the time of the process until time \(t>0\). Similar to the idea of  splitting the state vector of the system into purely discrete and purely continuous parts, we also separate \(Z(t)=(D(t),C(t))^{T}\) into purely discrete and continuous parts corresponding to the reaction counters of the slow, \(D(t) \in \mathbb{N}^{L}\), and the fast reaction set, \(C(t) \in \mathbb{R}^{R-L}_{\geq 0} \)  such that \(D_i(t)=Z_i(t),i \in \mathcal{D},\) and  \(C_j(t)=Z_j(t),j \in \mathcal{C}\).  We also separate the stoichiometric  vectors such that  \(\mu_{i}^{D}=\mu_i,\, i \in \mathcal{D},\)  and \(\mu_{j}^{C}=\mu_{j},\, j \in \mathcal{C} \). 
	
	By using  Equation (\ref{eq:6}), we will define reaction counters as follows: 
	\begin{eqnarray*}
		D_i(t)&=&\xi_{i} \Big( \displaystyle \int_{0}^{t} \alpha_i(Y(s)) ds\Big)=\xi_{i} \Big( \displaystyle \int_{0}^{t} \tilde{\alpha}_i(D(s),C(s)) ds\Big), \quad \quad i \in \mathcal{D} , \\
		C_j(t)&=&\displaystyle \int_{0}^{t} \alpha_{j}(Y(s))ds+W_j\Big(\int_{0}^{t} \alpha_{j}(Y(s))ds \Big)  \\
		&=&\displaystyle \int_{0}^{t} \tilde{\alpha}_{j}(D(s), C(s))ds+W_j \Big(\int_{0}^{t} \tilde{\alpha}_{j}(D(s),C(s))ds\Big), \quad \quad j \in \mathcal{C},
	\end{eqnarray*}
	where 
	\begin{equation}
	\label{new_propensity}
	\alpha_k(y)=\alpha_k(y(0)+\displaystyle \sum_{i \in \mathcal{D}}  d_{i} \mu_{i}^{D}+\displaystyle \sum_{j \in \mathcal{C}} c_j \mu_j^{C})=\tilde{\alpha}_k(d,c), \, k=1,2,\ldots, R.
	\end{equation}
	It must be noted that if \(\tau_1,\: \tau_2\) denote the successive firing times of reactions from the slow group, then for
	\(\tau_1<t<\tau_2\),  \(C(t)\) satisfies the following equation
	\begin{equation}
	\label{reaction_counter_dif}
	C(t)=C(\tau_1)+\displaystyle \sum_{j \in \mathcal{C}} \Big( \displaystyle \int_{\tau_1}^{t} \tilde{\alpha}_{j}(d,C(s)) ds\Big)e_j+\displaystyle \sum_{j \in \mathcal{C}} W_j \Big( \displaystyle \int_{\tau_1}^{t} \tilde{\alpha}_{j}(d,C(s)) ds\Big)e_j,
	\end{equation}
	where \(d \) denotes the number of slow reactions fired until time \(\tau_1 > 0\). 
	
	The HME is the time derivative of the  joint probability density function \(p_t: \mathbb{N}^{L} \times \mathbb{R}^{R-L}_{ \geq 0} \rightarrow \mathbb{R}_{\geq 0} \) 
	\begin{equation}
	\label{jp2}
	p_t(d,c)dc= \mathrm{P}(D(t)=d, C(t)\in [c,c+dc] ). 
	\end{equation}
	Then, we can write \[p_t(d,c)=p_t(c\mid d) p_t(d),\]
	where 
	\begin{eqnarray*}
		p_t(c \mid d)\, dc &=&\mathrm{P}(C(t) \in [c,c+dc] \mid  D(t)=d)\\
		p_t(d)&=&\mathrm{P}(D(t)=d).
	\end{eqnarray*}
	To obtain  the evolution equation for  \(p_t(d,c)\), which is called  the HME, we need the following result whose details can be found in \cite{paw:67}.
	\newline
	\textsl{Result} : Let \(D(t) \in \mathbf{D} \subset \mathbb{N}^{L}\) be a discrete process  and 
	\(C(t) \in \mathbb{R}^{R-L}_{ \geq 0}\) be a continuous process. Define the joint probability density function as follows: 
	\[p_t(d,c) =p_t(c\mid d) p_t(d).\]
	\newline 
	Then, the time derivative of this joint probability function, which is referred to as  generalized Fokker-Planck equation (GFPE), has the following form 
	
	\begin{equation}
	\label{pr2}
	\frac{\partial }{\partial t }p_t(d,c)=\displaystyle \sum_{d' \in \mathbf{D}} a_{dd'}p_{t}(d',c)
	+\displaystyle \sum_{n_1,n_2,\ldots,n_{R-L}=1}^{ \infty} \left( \displaystyle \prod_{i=1}^{R-L} \frac{(-1)^{n_{i}} \frac{\partial^{n_{i}}}{\partial c_{i}^{n_{i}} }}{n_{i}!}\right) [A_{n_1,n_2,\ldots,n_{R-L}}p_t(d,c)],
	\end{equation}
	where 
	\[A_{n_1,n_2,\ldots,n_{R-L}} =\lim_{h \rightarrow 0} \frac{1}{h} \mathsf{E}[\prod_{i=1}^{R-L} \{C_{i}(t+h)-C_{i}(t)\}^{n_{i}} \mid  D(t)=d,C(t),D(t+h)=d] ,\]
	and 
	\begin{equation}
	\label{eq:100}
	a_{dd'}=\lim_{ h \rightarrow 0} \frac{1}{h}[\mathrm{P}(D(t+h)=d\mid D(t)=d', C(t))- \delta_{d d'}].
	\end{equation}
	where \(\mathsf{E}[C(t) \mid d]= \displaystyle {\int_{ \mathbb{R}_{ \geq 0}^{R-L}}} c \: p_t(c \mid d) \, dc\). 
	It is also proved that \(A_{n_{1},n_{2},\ldots,n_{R-L}}=0\) for all \(\displaystyle \sum_{i=1}^{R-L} n_{i} \geq 3\). This gives us 
	\begin{eqnarray}
	\displaystyle \sum_{n_1,n_2,\ldots,n_{R-L}=1}^{\infty}  \left( \displaystyle \prod_{i=1}^{R-L} \frac{(-1)^{n_{i}} \frac{\partial^{n_{i}}}{\partial c_{i}^{n_{i}}}}{n_{i}!} \right) [A_{n_1,n_2,\ldots,n_{R-L}}p_t(d,c)]&=&- \displaystyle \sum_{j=1}^{R-L} \frac{\partial}{\partial c_{j}} [B_{j} p_t(d,c)] \\  \label{bi}
	\nonumber &+& \frac{1}{2} \displaystyle \sum_{i,j=1}^{R-L} \frac{\partial^{2}}{\partial c_{i}   \partial c_{j}} [B_{ij}  p_t(d,c), ]
	\end{eqnarray}
	where 
	\[B_{j}= \lim_{h \rightarrow 0} \displaystyle \frac{1}{h} \mathsf{E}[(C_{j}(t+h)- C_{j}(t))\mid  D(t)=d, C(t), D(t+dt )=d] ,\]
	and
	\[B_{ij}= \lim_{h \rightarrow 0} \displaystyle \frac{1}{h} \mathsf{E}[\{C_{i}(t+h)- C_{i}(t)\} \{C_{j}(t+h)- C_{j}(t)\} \mid  D(t)=d, C(t), D(t+dt )=d].\] 
	\begin{theorem}
		Let \(Z(t)=\{ D(t),C(t)\}\) be a joint reaction counting process where \(D(t)\) is a discrete random process with states \(d \in \mathbf{D} \subset \mathbb{N}^{L}, \: L >0,\)
		and \(C(t)\) is a continuous random process with states \(c \in \mathbb{R}^{R-L}_{\geq 0}\), \(R-L>0\). Define \(Y\) as a multi-scale process whose state vector is given in  Equation (\ref{eq:6}). Then, the joint counting probability density function given in  Equation (\ref{jp2}) satisfies the following  GFPE, which is referred to as the HME in the present paper. 
		
		\begin{eqnarray}
		\label{master_main} \displaystyle \frac{\partial}{\partial t} p_{t}(d,c)&=&\displaystyle \sum_{i \in \mathcal{D}} \Big(\tilde{\alpha}_{i}(d-\bar{e}_i,c) p_{t}(d-\bar{e}_i,c)-\tilde{\alpha}_{i}(d,c)p_t(d,c)\Big) \\
		&-&\nonumber \displaystyle \sum_{j \in \mathcal{C}}\displaystyle \frac{\partial }{\partial c_j} ( \tilde{\alpha}_{j}(d,c) p_t(d,c))
		+\frac{1}{2} \displaystyle  \sum_{j \in \mathcal{C}} \frac{\partial^{2}}{\partial c_{j}^{2}} ( \tilde{\alpha}_{j}(d,c)p_t(d,c)).
		\end{eqnarray}
	\end{theorem}
	\begin{proof}
		By using  Equation (\ref{pr2}) and  Equation  (\ref{bi}), we obtain
		\begin{equation}
		\frac{\partial}{\partial t} p_{t}(d,c)=\displaystyle \sum_{d' \in \mathbf{D}} a_{dd'}p_t(d',c) 
		- \displaystyle \sum_{j \in \mathcal{C}}\frac{\partial}{\partial c_{j}} (B_j p_{t}(d,c)) +\frac{1}{2}  \displaystyle \sum_{i,j \in \mathcal{C}}\frac{\partial^2}{ \partial c_i \partial c_j} (B_{ij} p_{t}(d,c)).   \label{hybrid_master}
		\end{equation}
		Now, let's  focus on the first summand on the right hand-side of  Equation (\ref{hybrid_master}), which can be rewritten as follows:
		\begin{equation}
		\label{a_dd}
		\sum_{d' \in \mathbf{D}} a_{dd'} p_{t}(d',c)=\displaystyle \sum_{\scriptsize{\begin{array}{c}
				d' \in \mathbf{D}\\
				d \neq d'
				\end{array}}}a_{dd'}p_{t}(d',c)+ a_{dd}p_{t}(d,c).
		\end{equation}
		Using Equation (\ref{eq:100}) gives  
		\[a_{dd}= \lim_{h \rightarrow 0} \frac{1}{h} [\mathrm{P}(D(t+h)=d\mid D(t)=d,C(t))-1],\]
		which can be reformulated as follows
		\[a_{dd}= \lim_{h \rightarrow 0} \frac{1}{h} [ -\displaystyle \sum_{\scriptsize{\begin{array}{c}
				d' \in \mathbf{D}\\
				d \neq d'
				\end{array}}} \mathrm{P}(D(t+h)=d'\mid  D(t)=d,C(t)) ]. \]
		By using this representation, we can rewrite  Equation (\ref{a_dd}) in the following form
		\[\displaystyle \sum_{d' \in \mathbf{D}} a_{dd'}p_{t}(d',c) = \displaystyle \sum_{\scriptsize{\begin{array}{c}
				d' \in \mathbf{D}\\
				d \neq d'
				\end{array}}} \Big( a_{dd'} p_{t}(d',c)- a_{d'd} p_{t}(d,c) \Big) .\]
		In our multi-scale process, we have \(L\) slow reactions, and one firing of the reaction \(R_j\) in this set  updates \(d\) to \(d+\bar{e}_j\). Starting from \(d\), the system can jump to \(d'=d+\bar{e}_j\), meaning that \(a_{d+\bar{e}_j,d}=\tilde{\alpha}_j(d,c)\). In the same vein, to reach \(d\), the system must supervene on \(d-\bar{e}_j\), by definition  \(a_{d,d-\bar{e}_j}= \tilde{\alpha}_{j}(d-\bar{e}_j,c).\) As a result, we obtain the desired summand as follows: 
		\begin{equation}
		\label{summand_1}
		\sum_{d' \in \mathbf{D}} a_{dd'} p_t(d',c)=\displaystyle \sum_{i \in \mathcal{D}}\Big(\tilde{\alpha}_i(d-\bar{e}_i,c) p_t(d-\bar{e}_i,c)-\tilde{\alpha}_i(d,c)p_t(d,c)\Big).
		\end{equation} 
		Now, we can concentrate on the second and the third summands of  Equation (\ref{hybrid_master}). Jump diffusion approximation is based on the idea that between two successive firing times of the slow reactions, the fast reactions continue to fire. Hence, the state vector and  also the reaction counting process of the fast reaction set  will satisfy diffusion processes (see  Equation \ref{diffusion},\ref{reaction_counter_dif}). Therefore, \(B_j\) and \(B_{ij}\) values have the forms \cite{gill:80,gill:02,kam:82}
		\begin{equation*}
		B_j= \displaystyle \sum_{k \in \mathcal{C}} e_{jk} \tilde{\alpha}_{k}(d,c), \quad B_{ij}=\displaystyle \sum_{k \in \mathcal{C}} e_{ik} e_{jk} \tilde{\alpha}_{k}(d,c). 
		\end{equation*}
		Substitution of   \(B_j\) and \(B_{ij}\) values and Equation  (\ref{summand_1})  into Equation (\ref{hybrid_master}) gives 
		\begin{eqnarray*}
			\nonumber \displaystyle \frac{\partial}{\partial t} p_{t}(d,c)&=&\displaystyle \sum_{i \in \mathcal{D}} \Big(\tilde{\alpha}_{i}(d-\bar{e}_i,c) p_{t}(d-\bar{e}_i,c)-\tilde{\alpha}_{i}(d,c)p_t(d,c)\Big) \\
			&-&\nonumber \displaystyle \sum_{j \in \mathcal{C}}\displaystyle \frac{\partial }{\partial c_j} \Big(\displaystyle \sum_{k \in \mathcal{C}} e_{jk} \tilde{\alpha}_{k}(d,c) \Big) + \frac{1}{2}  \displaystyle \sum_{i,j \in \mathcal{C}}\frac{\partial^2}{ \partial c_i \partial c_j} ( \displaystyle \sum_{k \in \mathcal{C}} e_{ik} e_{jk} \tilde{\alpha}_{k}(d,c) p_{t}(d,c)) \\
			&=&\displaystyle \sum_{i \in \mathcal{D}} \Big(\tilde{\alpha}_{i}(d-\bar{e}_i,c) p_{t}(d-\bar{e}_i,c)-\tilde{\alpha}_{i}(d,c)p_t(d,c)\Big)-\nonumber \displaystyle \sum_{j \in \mathcal{C}}\displaystyle \frac{\partial }{\partial c_j} (\tilde{\alpha}_{j}(d,c) p_{t}(d,c) ) + \frac{1}{2}  \displaystyle \sum_{j \in \mathcal{C}}\frac{\partial^2}{ \partial c_j^2} ( \tilde{\alpha}_{j}(d,c) p_{t}(d,c)),
		\end{eqnarray*}
		which completes the proof. 
	\end{proof}
	Based on the properties of the joint counting probability density function, we can write 
	\[p_t(d,c)=p_t(c\mid d)p_t(d). \]
	Since we partition reaction counters into two subsets, we will also decompose  the propensity  functions. Using mass action kinetics to compute propensities is very popular, and for this large class we partition the propensity function of the reaction \(R_k\), \(\tilde{\alpha}_k(d,c)\), \(k=1,2,\ldots,R\), as follows: 
	\begin{eqnarray*}
		\tilde{\alpha}_k(d,c)&=&\alpha_{k}(y(0)+\displaystyle \sum_{i \in \mathcal{D}}d_{i}\mu_{i}^{D}+ \displaystyle \sum_{j  \in \mathcal{C}}c_{j} \mu_{j}^{C})
		= \kappa_{k} \displaystyle \prod_{s=1}^{M}(y_s(0)+\displaystyle \sum_{i \in \mathcal{D}}d_i \mu_{si}^{D}+\displaystyle \sum_{j \in \mathcal{C}} c_j \mu_{sj}^{C})^{r_{sk}}\\
		&=&\kappa_{k} \displaystyle \prod_{s=1}^{M}(\beta_s(d)+\gamma_s(c))^{r_{sk}}
		= \kappa_{k} \displaystyle \prod_{s=1}^{M} \displaystyle \sum_{n=0}^{r_{sk} } \binom{r_{sk}}{n} \beta_{s}^{n}(d) \gamma_{s}^{r_{sk}-n}(c)
		= K \kappa_{k} \displaystyle \prod_{s=1}^{M} \displaystyle \sum_{n=0}^{r_{sk}} \beta_s^{n}(d) \gamma_{s}^{r_{sk}-n}(c),
	\end{eqnarray*} 
	where \(\beta_s(d)=y_s(0)+\displaystyle \sum_{i  \in \mathcal{D}}d_{i} \mu_{si}^{D}\), \(\gamma_{s}(c)=\displaystyle \sum_{j \in \mathcal{C}} c_{j}  \mu_{sj}^{C}\) and \(K\) is a real constant that will be ignored to simplify the notation for the reader. Based on this representation, the HME given in  Equation (\ref{master_main}) can be rewritten in the following form
	\begin{eqnarray}
	\displaystyle \frac{\partial}{\partial t} p_{t}(d,c)&=&\displaystyle \sum_{i \in \mathcal{D}}\displaystyle \prod_{s=1}^{M} \displaystyle \sum_{n=0}^{r_{si}}   \kappa_{i} \Big(\beta_{s}^{n}(d-\bar{e}_{i}) \gamma_{s}^{r_{si}-n}(c) p_{t}(d-\bar{e}_{i},c)-  \beta_{s}^{n}(d) \gamma_{s}^{r_{si}-n}(c)  p_{t}(d,c)  \Big) \label{prop} \\
	\nonumber &-&\displaystyle \sum_{j  \in \mathcal{C}}  \displaystyle \frac{\partial }{\partial c_j} \Big(
	\kappa_j \prod_{s=1}^{M} \displaystyle \sum_{n=0}^{r_{sj}} \beta_{s}^{n}(d) \gamma_{s}^{r_{sj}-n}(c) p_{t}(d,c) \Big)  +\frac{1}{2} \displaystyle \sum_{ j \in \mathcal{C}} \frac{\partial^{2}}{ \partial c_{j}^{2}} \Big(  \kappa_j  \prod_{s=1}^{M} \displaystyle \sum_{n=0}^{r_{sj}}\beta_{s}^{n}(d) \gamma_{s}^{r_{sj}-n}(c)     p_t(d,c)\Big). 
	\end{eqnarray}
	Let \(f(d): \mathbf{D} \rightarrow  \mathbb{R}\) and \(g(c):\mathbb{R}_{\geq 0}^{R-L}  \rightarrow  \mathbb{R} \) be any functions of \(d\) and \(c\) variables, respectively. To simplify the notation, we introduce one step operator 
	in the following form 
	\[ \mathcal{F}^{\bar{e}_i} \Big( f(d) g(c)\Big)= f(d+ \bar{e}_i) g(c), i \in \mathcal{D}.\]
	Based on this  representation, we define
	\begin{eqnarray*}
		\Gamma(\beta(d), \gamma(c))&=& \displaystyle \sum_{i \in \mathcal{D}} ( \mathcal{F}^{-\bar{e}_i} -I) \Big( \kappa_i \displaystyle \prod_{s=1}^{M}   \displaystyle \sum_{n=0}^{r_{si}} \beta_{s}^{n}(d) \gamma_{s}^{r_{si}-n}(c) \Big) \\
		&-& \displaystyle \sum_{j \in \mathcal{C}} \frac{ \partial }{ \partial c_j}  \Big(  \kappa_j   \prod_{s=1}^{M}   \displaystyle \sum_{n=0}^{r_{sj}} \beta_{s}^{n}(d) \gamma_{s}^{r_{sj}-n}(c) \Big) +\frac{1}{2} \displaystyle \sum_{j \in \mathcal{C}} \frac{ \partial^2 }{ \partial c_j^2} \Big(  \kappa_j   \prod_{s=1}^{M}   \displaystyle \sum_{n=0}^{r_{sj}} \beta_{s}^{n}(d) \gamma_{s}^{r_{sj}-n}(c) \Big). 
	\end{eqnarray*}
	Then, we can write  Equation (\ref{prop}) in the following form 
	\begin{equation}
	\label{master_main_operator}
	\frac{ \partial }{ \partial t} p_t(d,c) = \Gamma(\beta(d), \gamma(c)) p_t(d,c).
	\end{equation}
	In the rest of the study, we will assume that \(p_t(d,c)\) is zero at \(c=0\), \(c=\infty\) \cite{gts:11, rh:89}. 
	In the folllowing section, we will explain how we obtain the solution of this HME.

	\section{Solution of the Hybrid Master Equation}
	\label{solution_hme}
	To obtain the joint counting probability density function, \(p_t(d,c)\), described by the HME given in Equation (\ref{master_main_operator}), we will approximate the process \(C(t) \mid D(t)\) using its moments. Solving a maximum entropy problem for each conditional moment will produce the conditional probability function, \(p_t(c \mid d)\). The multiplication of \(p_t(c\mid d)\) with the marginal probabilities of the remaining discrete states, i.e., \(p_t(d)= \displaystyle \int_{\mathbb{R}_{ \geq 0}^{R-L}} p_t(d,c) \, dc,\) will give us the desired joint probability density function \(p_t(d,c)\) . 
	
	In the rest of the study, time dependent conditional means and the centered conditional moments of the process \(C(t)\mid D(t)\) will be denoted by
	\begin{eqnarray*}
		\mathsf{E}_{t}[C_m\mid d]&=&\displaystyle \int_{\mathbb{R}_{ \geq 0}^{R-L}}  c_m p_{t}(c\mid d) \,dc, \quad  m \in \mathcal{C},\\
		\mathsf{E}_{t}[\tilde{C}^{M}\mid d]&=& \displaystyle \int_{\mathbb{R}_{ \geq 0}^{R-L}}      \displaystyle \prod_{j \in \mathcal{C}}\tilde{c}_{j}^{M_j}  p_t(c \mid d)dc, 
	\end{eqnarray*}
	where \(\tilde{c}=c-\mathsf{E}_{t}[C\mid d]\), \(M=(M_1,M_2,\ldots,M_{R-L})^{T} \in \mathbb{N}^{R-L}\). 
	Now, based on the study \cite{hwkt:13}, we want to construct  a differential equation system to obtain  \(p_{t}(d)\), \(\mathsf{E}_t[ C_m\mid  d]\), \(\mathsf{E}_{t}[\tilde{C}^{M}\mid d ]\). To construct this system, we will need the following Lemma \cite{eng:06,hwkt:13}. 
	\begin{lemma} \label{lemma1}
		Let \(F_t: \mathbb{R}^{R-L}_{\geq 0} \longrightarrow \mathbb{R}\)  be a polynomial function of \(c\), and \(p_{t}(d,c)\) satisfy differential Equation (\ref{master_main_operator}). Assume that sufficiently many moments of \(p_{t}(d,c)\) with respect to \(c\) exist, and the joint counting probability density vanishes at \(c=0\) and \(c=\infty\). Define the following conditional mean
		\[\mathsf{E}_t[F_t(C)\mid d ]=\displaystyle \int_{\mathbb{R}_{ \geq 0}^{R-L}} F_t(c) p_{t}(c\mid d) \,dc. \]
		Then, 
		\[\frac{\partial}{\partial t}(\mathsf{E}_t[ F_t(C)\mid d]p_t(d))= \tilde{\Gamma}( \beta(d), F_t(c)\gamma(c)) p_t(d)+\mathsf{E}_t[ \frac{ \partial }{\partial t} F_t(C) \mid d]p_t(d),\]
		where
		\begin{eqnarray*}
			\tilde{\Gamma}(\beta(d), F_t(c)\gamma(c))&=& \displaystyle \sum_{i \in \mathcal{D}} ( \mathcal{F}^{-\bar{e}_i} -I) \Big( \kappa_i \displaystyle \prod_{s=1}^{M}   \displaystyle \sum_{n=0}^{r_{si}} \beta_{s}^{n}(d) \mathsf{E}_{t}[ F_t(C) \gamma_{s}^{r_{sj}-n}(C) \mid d] \Big)\\
			&+& \displaystyle \sum_{j \in \mathcal{C}}  \kappa_{j} \displaystyle \prod_{s=1}^{M} \displaystyle \sum_{n=0}^{r_{sj}}
			\beta_{s}^{n}(d) \mathsf{E}_{t}[ \gamma_{s}^{r_{sj}-n}(C) \frac{\partial}{\partial c_j} F_t(C) \mid d]+\frac{1}{2} 
			\displaystyle \sum_{j \in \mathcal{C}} \kappa_{j} \prod_{s=1}^{M}   \displaystyle \sum_{n=0}^{r_{sj}} \mathsf{E}_t[ 
			\gamma_{s}^{r_{sj}-{n}}(C)  \frac{\partial^{2}}{\partial c_j^{2}} F_t(C) \mid d].
		\end{eqnarray*}
	\end{lemma}
	\begin{proof}
		The proof of the Lemma can be found in Appendix  \ref{proof_lemma1}. 
	\end{proof}
	
	When \(F_t(c)=1\) in Lemma \ref{lemma1}, we obtain the time derivative of the marginal probability \(p_t(d)\), which is given in the following proposition. 
	\begin{Proposition}
		\label{prop1}
		\begin{equation}
		\label{marjinal}
		\frac{\partial}{\partial t} p_{t}(d)=  \displaystyle \sum_{i \in \mathcal{D}} ( \mathcal{F}^{-\bar{e}_i} -I) \Big( \kappa_i \displaystyle \prod_{s=1}^{M}   \displaystyle \sum_{n=0}^{r_{si}} \beta_{s}^{n}(d) 
		\mathsf{E}_t[\gamma_{s}^{r_{si}-n}(C)\mid d]p_t(d)\Big) 
		\end{equation}
	\end{Proposition}
	The strategy of our method is to obtain \(p_{t}(d)\) and \(p_{t}(c\mid d)\) separately and construct the joint probability function using the equality \(p_t(d,c)=p_t(c\mid d)p_{t}(d)\). To obtain the conditional probability \(p_t(c\mid d)\), we will use  evolution equations of the  conditional means \(\mathsf{E}_t[C_m\mid d]\), \(m \in \mathcal{C}\),  and the centered conditional moments 
	\( \mathsf{E}_t[ \tilde{C}^{M}\mid d ]\), \(M \in \mathbb{N}^{R-L}\), which are the functions of   \(p_t(d)\), \(\mathsf{E}_t[C_m\mid d], \: \mathsf{E}_t[ \tilde{C}^{M}\mid d ]\). Equation  (\ref{marjinal}) will be the first equation of our system. It must be noted that differential equation defining the marginal probability only depends on the slow reactions. To solve this differential equation, we need to reformulate the unknown conditional means \(\mathsf{E}_t[ \gamma_s^{r_{si}-n}(C)\mid d]\)  through the known \(\mathsf{E}_t[C_m\mid d]\), \(\mathsf{E}_t[ \tilde{C}^{M}\mid d ]\).  The details of this transformation can be found in Appendix  \ref{cons_mean}.  
	
	In the following proposition, we will obtain the time evolution equation for the conditional means \(\mathsf{E}_t[C_m\mid d], \: m \in \mathcal{C}\). 
	\begin{Proposition}
		\label{prop_moment}
		\begin{eqnarray}
		p_t(d) \frac{\partial}{\partial t } \mathsf{E}_t[C_m\mid d]  &=&\displaystyle   \sum_{i \in \mathcal{D}}  
		( \mathcal{F}^{-\bar{e}_i} -I) \Big( \kappa_i \displaystyle \prod_{s=1}^{M}   \displaystyle \sum_{n=0}^{r_{si}} \beta_{s}^{n}(d) \mathsf{E}_t[C_m \gamma_{s}^{r_{si}-n}(C)\mid d] p_t(d) \Big)  \label{moment} \\
		\nonumber &+& \displaystyle\sum_{j \in \mathcal{C}}  \kappa_j \displaystyle \prod_{s=1}^{M} \displaystyle \sum_{n=0}^{r_{sj}} \beta_s^{n}(d) \mathsf{E}_t[\gamma_{s}^{r_{sj}-n}(C) \delta_{jm} \mid d]p_t(d)- \mathsf{E}_t[C_m\mid d] \frac{\partial}{\partial t }  p_t(d)  
		\end{eqnarray}
		where  \(\delta_{jm}\) is the Kronecker delta function. 
	\end{Proposition}
	\begin{proof}
		The proof of the proposition can be found in Appendix  \ref{proof_prop_moment}.
	\end{proof}
	In the following proposition, we will obtain \(p_t(d) \displaystyle \frac{\partial}{\partial t} \mathsf{E}_t[\tilde{C}^{M}\mid d]\).
	\begin{Proposition}
		\label{prop_var}
		\begin{eqnarray}
		\nonumber p_t(d) \displaystyle \frac{\partial}{\partial t} \mathsf{E}_t[\tilde{C}^{M}\mid d]&=&  \displaystyle   \sum_{i \in \mathcal{D}}  
		( \mathcal{F}^{-\bar{e}_i} -I) \Big( \kappa_i \displaystyle \prod_{s=1}^{M}   \displaystyle \sum_{n=0}^{r_{si}} \beta_{s}^{n}(d) \mathsf{E}_t[\tilde{C}^{M} \gamma_{s}^{r_{si}-n}(C)\mid d] p_t(d) \Big)\\
		&+& \displaystyle \sum_{j \in \mathcal{C}}\kappa_j \displaystyle \prod_{s=1}^{M} \displaystyle \sum_{n=0}^{r_{sj}} \beta_{s}^{n}(d) \mathsf{E}_t[M_j\gamma_{s}^{r_{sj}-n}(C)  \tilde{C}^{M-e_j}\mid d]p_t(d) \label{final}\\
		\nonumber &+& \frac{1}{2}  \displaystyle \sum_{j \in \mathcal{C}} \kappa_j   \displaystyle \prod_{s=1}^{M} \displaystyle \sum_{n=0}^{r_{sj}} \beta_s^{n}(d) \mathsf{E}_t[M_j (M_j-1) \gamma_{s}^{r_{sj}-n}(C)  \tilde{C}^{M-2e_j} \mid d ]  p_t(d) \\
		\nonumber &-& \displaystyle \sum_{j  \in \mathcal{C}}M_j \mathsf{E}_{t}[\tilde{C}^{M-e_j}\mid d] p_t(d) \frac{\partial}{\partial t }\mathsf{E}_t[C_j\mid d]- \mathsf{E}_t[\tilde{C}^{M}\mid d] \frac{\partial}{\partial t} p_t(d) .
		\end{eqnarray}
	\end{Proposition}
	\begin{proof}
		The proof of the theorem can be found in Appendix Section  \ref{proof_prop_var}. 
	\end{proof}
	Up to this section, we have obtained the time derivatives of the marginal probabilities as well as those of  the conditional means  and the centered conditional moments. These three equations will give us the following differential equation system. 
	\begin{theorem}
		Let \(p_t(d,c)=p_t(c\mid d) p_t(d)\) satisfy  Equation (\ref{master_main}). Then, the time derivative of \( p_t(d)\), \(\mathsf{E}_t [C_m\mid d]\), \(m \in \mathcal{C}\) and \(\mathsf{E}_t[\tilde{C}^{M}\mid d] \), \(M=(M_1,M_2, \ldots, M_{R-L})\), satisfies the following system 
		\begin{eqnarray}
		\nonumber \frac{\partial}{\partial t} p_{t}(d)&=& \displaystyle \sum_{i \in \mathcal{D}} ( \mathcal{F}^{-\bar{e}_i} -I) \Big( \kappa_i \displaystyle \prod_{s=1}^{M}   \displaystyle \sum_{n=0}^{r_{si}} \beta_{s}^{n}(d) 
		\mathsf{E}_t[\gamma_{s}^{r_{si}-n}(C)\mid d]p_t(d)\Big)  \\
		\nonumber  p_t(d) \frac{\partial}{\partial t } \mathsf{E}_t[C_m\mid d]  &=&\displaystyle   \sum_{i \in \mathcal{D}}  
		( \mathcal{F}^{-\bar{e}_i} -I) \Big( \kappa_i \displaystyle \prod_{s=1}^{M}   \displaystyle \sum_{n=0}^{r_{si}} \beta_{s}^{n}(d) \mathsf{E}_t[C_m \gamma_{s}^{r_{si}-n}(C)\mid d] p_t(d) \Big)  \\
			\nonumber &+& \displaystyle\sum_{j \in \mathcal{C}}  \kappa_j \displaystyle \prod_{s=1}^{M} \displaystyle \sum_{n=0}^{r_{sj}} \beta_s^{n}(d) \mathsf{E}_t[\gamma_{s}^{r_{sj}-n}(C) \delta_{jm}(C)\mid d]p_t(d)- \mathsf{E}_t[C_m\mid d] \frac{\partial}{\partial t }  p_t(d)  \\
		\nonumber p_t(d) \displaystyle \frac{\partial}{\partial t} \mathsf{E}[\tilde{C}^{M}\mid d]&=&   \displaystyle   \sum_{i \in \mathcal{D}}  
		( \mathcal{F}^{-\bar{e}_i} -I) \Big( \kappa_i \displaystyle \prod_{s=1}^{M}   \displaystyle \sum_{n=0}^{r_{si}} \beta_{s}^{n}(d) \mathsf{E}_t[\tilde{C}^{M} \gamma_{s}^{r_{si}-n}(C)\mid d] p_t(d) \Big)\\
	 &+& \displaystyle \sum_{j \in \mathcal{C}}\kappa_j \displaystyle \prod_{s=1}^{M} \displaystyle \sum_{n=0}^{r_{sj}} \beta_{s}^{n}(d) \mathsf{E}_t[M_j\gamma_{s}^{r_{sj}-n}(C)  \tilde{C}^{M-e_j}\mid d]p_t(d)  \label{total_system}  \\
		\nonumber &+& \frac{1}{2}  \displaystyle \sum_{j \in \mathcal{C}} \kappa_j   \displaystyle \prod_{s=1}^{M} \displaystyle \sum_{n=0}^{r_{sj}} \beta_s^{n}(d) \mathsf{E}_t[M_j (M_j-1) \gamma_{s}^{r_{sj}-n}(C)  \tilde{C}^{M-2e_j} \mid d ]  p_t(d) \\
		\nonumber &-& \displaystyle \sum_{j  \in \mathcal{C}}M_j \mathsf{E}_{t}[\tilde{C}^{M-e_j}\mid d] p_t(d) \frac{\partial}{\partial t }\mathsf{E}_t[C_j\mid d]- \mathsf{E}_t[\tilde{C}^{M}\mid d] \frac{\partial}{\partial t} p_t(d),
		\end{eqnarray}
		where  \(\delta_{jm}\) is kronecker delta function. Also,  \(\mathcal{F}^{\bar{e}_i}\) is a one step operator as follows:
		\[ \mathcal{F}^{\bar{e}_i} \Big( f(d) g(c)\Big)= f(d+ \bar{e}_i) g(c), i \in \mathcal{D}.\]
		where \(f(d): \mathbf{D} \rightarrow  \mathbb{R}\) and \(g(c):\mathbb{R}_{\geq 0}^{R-L}  \rightarrow  \mathbb{R} \) be any functions of \(d\) and \(c\) variables, respectively. 
	\end{theorem}
	
	In the following section, we will explain the details of the maximum entropy method which will be used to construct the conditional probability distribution \(p_t(c \mid d)\). 
	\section{ Maximum Entropy }
	\label{entropy}
Assume that we want to obtain the solution of the HME under consideration at a specific time point 
	\(\tau >0\).  Solving the ODE system in  (\ref{total_system})  gives \(p_{\tau}(d) \), \(\mathsf{E}_{\tau}[C_m \mid d] \),\( \mathsf{E}_{\tau} [\tilde{C}^{M} \mid d]\), \(m \in \mathcal{C}\), \(M=(M_1,M_2,\ldots,M_{R-L})^{T} \in \mathbb{N}^{R-L}\) values for the system of interest. Although the marginal probabilities, \(p_{\tau}(d)\), can directly be obtained from the ODE system, we still do not know the corresponding conditional probability density function, \(p_{\tau}(c \mid d)\), which will be used to construct the joint probability, \(p_{\tau}(d,c)\), solving the corresponding HME. 
	
	To estimate  the unknown conditional probability density functions using its moments, we will use the maximum entropy approach proposed by Shannon \cite{sha:48}. 
	Assume that we have a state space \( \Omega=\mathbf{D} \times \mathbb{R}_{\geq 0}^{R-L}\) and our goal is to estimate the unknown probability density function \( p_{\tau}: \Omega \rightarrow \mathbb{R}_{\geq 0}\). Let 
	\[ \mathcal{S}_{\tau}^{M}=\displaystyle \int_{\mathbb{R}_{\geq 0}^{R-L}} \prod_{j \in \mathcal{C}} c_{j}^{M_j} p_{\tau}(c \mid d)dc, \quad M=(M_1,M_2,\ldots, M_{R-L}) \in \mathbb{N}^{R-L},\]
	denote the moments of the joint probability density function at time point \(\tau\). It must be noted that when 
	\(M=e_m\), we obtain \(\mathsf{E}_{\tau} [C_{m} \mid d]\). To guarantee that \(p_{\tau}(c \mid d)\) is a probability function, we must impose the condition \(\mathcal{S}_{\tau}^{0}=1\). Then, the approximation for the conditional probability density function \(p_{\tau}(c\mid d)\) will be obtained solving the following constrained convex optimization problem

	\begin{eqnarray*}
		\mathrm{Minimize} \quad   && \displaystyle \int_{  \mathbb{R}^{R-L}_{\geq 0}}  p_{\tau}{(c \mid d) } \ln ( p_{\tau} (c \mid d))dc \\\mathrm{Subject \: to}  \quad  \mathcal{S}_{\tau}^{0}&=&   \displaystyle \int_{  \mathbb{R}^{R-L}_{\geq 0}}  p_{\tau}{(c \mid d) } =1 \\
		\mathcal{S}_{\tau}^{e_m}&=&  \mathsf{E}_{\tau}[C_m \mid d] =\displaystyle \int_{  \mathbb{R}^{R-L}_{\geq 0}} c_m  p_{\tau}{(c \mid d) }  \\
		\mathcal{S}_{\tau}^{M}&=&  \mathsf{E}_{\tau}[C^M \mid d] =\displaystyle \int_{  \mathbb{R}^{R-L}_{\geq 0}} \displaystyle \prod_{ j \in \mathcal{C}}c_j^{M_j}   p_{\tau}{(c \mid d) }  \\
	\end{eqnarray*}
	Let \(N\) be the number of moment constraints and \(M^k\), \(k=0,1,\ldots, N\), denote different choices of vectors \newline \(M^k=(M_{1}^{k},M_{2}^{k}, \ldots, M_{R-L}^{k} )\in \mathbb{N}^{R-L}\). To impose the conditions given above, we will have \(M^0=0\), \(M^{j}=e_j\), \(j=1,2,\ldots, R-L\). Then, the solution of this constrained optimization problem can be obtained maximizing the following Lagrange function 
	\[\mathcal{L}(p_{\tau}(c\mid d), \lambda(\tau))= - \displaystyle \int_{\mathbb{R}^{R-L}_{\geq 0}}  p_{\tau}{(c \mid d) } \ln ( p_{\tau} (c \mid d))dc + \displaystyle \sum_{k=0}^{N} \lambda_{k}(\tau) \Big(  \displaystyle \int_{\mathbb{R}^{R-L}_{\geq 0}}  \displaystyle \prod_{ j \in \mathcal{C}}c_j^{M_{j}^{k}} p_{\tau}{(c \mid d) }dc -\mathcal{S}_{\tau}^{M^k}\Big), \]
	where \(\lambda_k \in \mathbb{R}\), \(k=1,2,\ldots,N\) are referred to as Lagrange multipliers. Taking the derivative of \(\mathcal{L}(p_{\tau}(c\mid d), \lambda(\tau))\) with respect to \(p_{\tau}(c \mid d)\)  will give the approximate solution of the conditional probability density for \(p_{\tau}(c \mid d)\) in the following form
	\[ p_{\tau}^{*}(c \mid d)=argmax (\mathcal{L}(p_{\tau}(c\mid d), \lambda(\tau)))=\frac{1}{Z(N,\lambda(\tau))} \exp\Big(- \displaystyle \sum_{k=0}^{N} \lambda_k(\tau) \prod_{ j \in \mathcal{C}}c_j^{M_{j}^{k}}\Big),\]
	where \(Z(N,\lambda(\tau))\) is a normalization constant
	\cite{ab:10,bv:04,bre:13}. Now, we can obtain the approximate solution of the joint probability density function which solves the HME under consideration by multiplying the obtained conditional probability function \( p_{\tau}^{*}(c \mid d)\) with the marginal probability function \(p_{\tau}(d)\).
	\section{Application}
	\label{appl_sec}
	In this section of the present study, we will implement our proposed method to the following reaction system 
	\[R_{1} : 2 S_1\stackrel{\kappa_{1}} {\longrightarrow}  2 S_2, \quad R_{2}: S_2 \stackrel{\kappa_{2}} {\longrightarrow}  S_1.\]
	The  state vector of the system at time \(t \geq 0\) is defined by \(Y(t)=(Y_1(t),Y_2(t))^{T} \in \mathbb{Z}_{ \geq 0}^{2}\), where \(Y_i(t)\) denote the number of  molecules of species \(S_i\), \(i=1,2\). 
	
	The joint  probability density function, \(p_t(d,c) \), satisfies the following CME
	\begin{eqnarray}
	\label{birth_death_cme}
	\nonumber \displaystyle\frac{\partial}{\partial t} p_t(d,c)&=& \kappa_1(y_1(0)-2(d-1)+c) p_t(d-1,c)-\kappa_1(y_1(0)-2d+c) p_t(d,c) \\
	&-&\kappa_2(y_2(0)+2d-(c-1)) p_t(d,c)+ \kappa_2(y_2(0)+2d-c) p_t(d,c).
	\end{eqnarray}
	We separate   reactions and  stoichiometric vectors  as follows:
	\[ \mathcal{D}=\{1\}, \: \mathcal{C}=\{2\},\: \mu_1^{D}=(-2,2)^{T},\:  \mu_2^{C}=(1,-1)^{T}.\]
	Propensity functions of the reactions are assumed to be 
	\[ \tilde{\alpha}_1(d,c)= \kappa_1(\beta_1(d)+\gamma_1(c)), \: \tilde{\alpha}_2(d,c)= \kappa_2(\beta_2(d)+\gamma_2(c)),\]
	where 
	\[\beta_1(d)=y_1(0)-2d, \: \gamma_1(c)=c, \: \beta_2(d)=y_2(0)+2d, \gamma_2(c)=-c .\]
	Then, the HME for the joint probability density function, \(p_t(d,c) \), is defined as given below:
	\begin{eqnarray}
	\label{birth_death_hybrid_master}
	\nonumber \displaystyle\frac{\partial}{\partial t} p_t(d,c)&=& \kappa_1(y_1(0)-2(d-1)+c) p_t(d-1,c)-\kappa_1(y_1(0)-2d+c) p_t(d,c) \\
	&-&\displaystyle\frac{\partial}{\partial c} \Big( \kappa_2(y_2(0)+2d-c) p_t(d,c)\Big)+ \displaystyle \frac{1}{2}\displaystyle\frac{\partial^2}{\partial c^2} \Big( \kappa_2(y_2(0)+2d-c) p_t(d,c) \Big)
	\end{eqnarray}
	The system of differential equation defining the marginal probabilities, the conditional means and the centered conditional moments has the following form 
	\begin{eqnarray}
	\nonumber \displaystyle \frac{\partial}{\partial t} p_{t}(d)&=& \Big(\kappa_1(y_1(0)-2(d-1)) + \kappa_1 \mathsf{E}_t[C\mid d-1] \Big) p_t(d-1) 
	-\Big(\kappa_1(y_1(0)-2d)- \kappa_1 \mathsf{E}_t[C\mid d] \Big) p_t(d) \\
	\nonumber  p_t(d) \displaystyle \frac{\partial}{\partial t} \mathsf{E}_t[C\mid d]&=& \kappa_1(y_1(0)-2(d-1)) \mathsf{E}_t[C\mid d-1] p_t(d-1)- \kappa_1(y_1(0)-2d) \mathsf{E}_t[C\mid d] p_t(d) \\
	\nonumber  &+& \kappa_1 \mathsf{E}_t[C^2\mid d-1] p_t(d-1) -\kappa_1  \mathsf{E}_t[C^2\mid d] p_t(d) 
	+ \kappa_2(y_2(0)+2d) p_t(d)-  \kappa_2 \mathsf{E}_t[C\mid d] p_t(d) \\
	&-&\mathsf{E}_t[C\mid d] \displaystyle \frac{\partial}{\partial t} p_{t}(d)\\
	\nonumber  p_t(d)\displaystyle \frac{\partial}{\partial t} \mathsf{E}_t[\tilde{C}^{2}\mid d]&=&\kappa_1(y_1(0)-2(d-1)) \mathsf{E}_t[\tilde{C}^2\mid  d-1]p_t(d-1)-\kappa_1(y_1(0)-2d) \mathsf{E}_t[\tilde{C}^2\mid d]p_t(d) \\
	\nonumber  &+&\kappa_1 \mathsf{E}_t [\tilde{C}^2 C\mid d-1] p_t(d-1)-\kappa_1 \mathsf{E}_t [\tilde{C}^2 C\mid d] p_t(d) + \kappa_2(y_2(0)+2d) p_t(d) -\kappa_2  \mathsf{E}_t [C\mid d] p_t(d) \\
	\nonumber  &-&2 \kappa_2 \mathsf{E}_t[\tilde{C}C\mid d]p_t(d)- \mathsf{E}_t[\tilde{C}^{2}\mid d]  \displaystyle \frac{\partial}{\partial t} p_{t}(d).
	\end{eqnarray}
	Based on our previous discussions, we  get the following system of differential equation which will be referred to as the moment equation system of the HME in the rest of the study
	\begin{eqnarray}
	\nonumber \displaystyle \frac{\partial}{\partial t} p_{t}(d)&=&\Big(\kappa_1(y_1(0)-2(d-1)) + \kappa_1 \mathsf{E}_t[C\mid d-1] \Big) p_t(d-1) 
	-\Big(\kappa_1(y_1(0)-2d)- \kappa_1 \mathsf{E}_t[C\mid d] \Big) p_t(d)   \label{marginal_model} \\
	\nonumber p_t(d) \displaystyle \frac{\partial}{\partial t} \mathsf{E}_t[C\mid d]&=& \kappa_1(y_1(0)-2(d-1)) \mathsf{E}_t[C\mid d-1] p_t(d-1)- \kappa_1(y_1(0)-2d) \mathsf{E}_t[C\mid d] p_t(d) \\
	&+& \kappa_1 \mathsf{E}_t[\tilde{\Psi}^2\mid d-1] p_t(d-1)+ \kappa_1  (\mathsf{E}_t[C\mid d-1])^{2} p_t(d) 
	-  \kappa_1 \mathsf{E}_t[\tilde{C}^2\mid d] p_t(d) \label{mean_model} \\
	\nonumber  &-& \kappa_1  (\mathsf{E}_t[C\mid d])^{2} p_t(d)  + \kappa_2(y_2(0)+2d) p_t(d)-  \kappa_2 \mathsf{E}_t[C\mid d] p_t(d)-\mathsf{E}_t[C\mid d] \displaystyle \frac{\partial}{\partial t} p_{t}(d)\\
	\nonumber  p_t(d) \displaystyle \frac{\partial}{\partial t} \mathsf{E}_t[\tilde{C}^{2}\mid d]&=& \Big(\mathsf{E}_t[\tilde{\Psi}^2\mid  d-1]+  \{\mathsf{E}_t[C\mid  d-1]-\mathsf{E}_t[C\mid  d] \}^2 \Big) \kappa_1(y_1(0)-2(d-1))  p_t(d-1) \\
	\nonumber&-& \kappa_1(y_1(0)-2d)  \mathsf{E}_t[\tilde{C}^2\mid  d]p_t(d) +\kappa_1 \mathsf{E}_t[C\mid  d-1] \mathsf{E}_t[\tilde{\Psi}^2\mid  d-1] p_t(d-1)  \\
	\nonumber &+& 2 \kappa_1 \{\mathsf{E}_t[C\mid  d-1]-\mathsf{E}_t[C\mid  d] \} \mathsf{E}_t[\tilde{\Psi}^2\mid  d-1] p_t(d-1)  \label{moment_model}\\
	\nonumber&+&  \kappa_1 \{\mathsf{E}_t[C\mid  d-1]-\mathsf{E}_t[C\mid  d] \}^2   \mathsf{E}_t[C \mid  d-1] p_t(d-1) - \kappa_1  \mathsf{E}_t[C\mid  d] \mathsf{E}_t[\tilde{C}^2\mid  d] p_t(d)\\
	\nonumber&+& \kappa_2(y_2(0)+ 2d) p_t(d)-\kappa_2 \mathsf{E}_t[C\mid  d]p_t(d)-2\kappa_2 \mathsf{E}_t[\tilde{C}^2\mid  d] p_t(d)
	- \mathsf{E}_t[\tilde{C}^{2}\mid d]  \displaystyle \frac{\partial}{\partial t} p_{t}(d),
	\end{eqnarray}
	where \( \tilde{\Psi}=c-\mathsf{E}_t[C \mid d -1]\). 
	
	Substitution  \(\displaystyle \frac{\partial}{\partial t} p_{t}(d)\)  into \(p_t(d) \displaystyle \frac{\partial}{\partial t} \mathsf{E}_t[C\mid d]\), \( p_t(d) \displaystyle \frac{\partial}{\partial t} \mathsf{E}_t[\tilde{C}^{2}\mid d]\) will give a system of differential equations that is expressed only in terms of the marginal probabilities, the conditional means and the second centered moments. In our application, we close moment equations setting the third and the higher moments to zero. If \( p_t(d)=0\) , then we will not be able to obtain  \(\mathsf{E}_t[C \mid d], \mathsf{E}_t[\tilde{C}^{2}\mid d] \). To avoid this drawback, in \cite{hwkt:13}, the authors proposed a successful  initialization procedure. 
	
	Based on the fact that propensity functions must be non negative, we define the state space of the system as follows: 
	\[\Omega = \mathbf{D} \times \mathbf{C}=\{(d,c) \in  \mathbf{D} \times \mathbb{R}_{\geq 0}:  y_1(0)-2d+c \geq 0,  y_2(0)+2d-c \geq 0, \mathbf{D} \subset \mathbb{N}\}. \]
	To obtain each conditional probability density function by solving the corresponding convex optimization problem on the state space of interest, we use the CVX toolbox of the MATLAB  \cite{cvx}. When the size of \(\Omega\) is very high, the dimensionality of the optimization problem increases. Therefore, the CVX cannot produce accurate results. 
	
	To keep the dimension of the optimization problems small for the CVX, we construct state space iteratively using a similar strategy to the sliding window method \cite{wgmh:10}. In summary, our strategy is to solve the moment equation system of the HME using an appropriate discretization method. At each discretization step, we check the marginal probabilities. If they are higher than a given threshold, then we extend the state space of the variable \(d\). This procedure continues until the time point of the interest is reached. Finally, depending on the state space of the variable \(d\), we construct a state space for the variable \(c\). Now, we can explain the details of the method. 
	
	In the first step of this construction, we define a feasible subset \(\Omega^0\) of \(\Omega\), \(\Omega^0=\mathbf{D}^0 \times \mathbf{C}^0\), in which the dimension of the optimization problem is acceptable for  the CVX. To avoid the problem of having \(p_0(d)=0\), we choose an initial Poisson distribution, \(p_0(d,c)\), in the state space \(\Omega^0\) and compute the corresponding \(p_0(d)\), \(\mathsf{E}_0[C \mid d]\), \(\mathsf{E}_0[\tilde{C}^2 \mid d]\), which  will be considered  as the initial conditions for the moment equation of the system of the HME. 
	
	Assume that we want to obtain the conditional counting probability density at time point \(\tau >0\). Then, we approximate the solution of the moment equation system of the HME on \([0, \tau]\) time interval using a numerical method. We choose a discretization time step \(\Delta\) and define \(t_j=j\Delta\), \(j=0,1,\ldots, J\) such that \(t_0=0\), \(t_J=\tau\). As a result, we obtain subintervals \([t_j, t_{j+1}]\), \(j=0,1,\ldots,  J-1\). Let \(p^{j}(d)\), \(\mathsf{E}^{j}[C \mid d]\), \(\mathsf{E}^{j}[\tilde{C}^2 \mid d]\) represent the approximate solution of the ODE system given in  Equation  (\ref{mean_model}) and \(\mathbf{D}^{j}\) represents the state space of \(d\)  at time point \(t_j\). To construct \(\mathbf{D}^1\), we will solve the moment equation system of the HME using initial conditions \(p^{0}(d) \equiv p_0(d)\), \( \mathsf{E}^{0} [C\mid d] \equiv \mathsf{E}_{0} [C\mid d] \), \( \mathsf{E}^{0} [\tilde{C}^2\mid d] \equiv \mathsf{E}_{0} [\tilde{C}^2\mid d] \). Then, we will obtain \(p^{1}(d)\), \(\mathsf{E}^{1}[C\mid d]\), \(\mathsf{E}^{1}[\tilde{C}^2\mid d]\) values for each \(d\) variable in the state space 
	\[\mathbf{D}^{0}=\{d: \min(\mathbf{D}^{0}) \leq d \leq \max(\mathbf{D}^{0})\}.\]
	To extend \( \mathbf{D}^{0}\), we define  a threshold \( \varepsilon >0 \) and check the marginal probability 
	\(p^{1}_{max} \equiv p^{1}(max(\mathbf{D}^{0}))\). If \(p^{1}_{max}> \varepsilon \), then we extend \(\mathbf{D}^0\) as follows:
	\[\mathbf{D}^{1}=\{d: \min(\mathbf{D}^{0}) \leq d \leq \max(\mathbf{D}^{0})+1\}.\]
To approximate the solution of   Equation (\ref{mean_model}) at time point \(t_2\), we need to initialize the system on \(\mathbf{D}^{1}\). Although we know \(p^1(d)\), \(\mathsf{E}^{1}[C\mid d]\), \(\mathsf{E}^{1}[\tilde{C}^2\mid d]\) for 
	\(d \in \mathbf{D}^{0}\), we have to impose initial conditions for \(d=\max(\mathbf{D}^{0})+1\)
	
	\[p^1( \max(\mathbf{D}^{0})+1 ) = \frac{\displaystyle \sum_{d  \in \mathbf{D}^0}  p^1(d) }{\mid  \mathbf{D}^0 \mid +1}, \quad  \mathsf{E}^{1}[C \mid  \max(\mathbf{D}^{0})+1 ]= \mathsf{E}^{1}[ C \mid  \max(\mathbf{D}^{0})], \quad  \mathsf{E}^{1}[\tilde{C}^2 \mid \max(\mathbf{D}^{0})+1] = \mathsf{E}^{1}[ \tilde{C}^2  \mid \max(\mathbf{D}^{0})],\] 
	where   \(\mid  \mathbf{D}^0 \mid \) denotes the cardinality of the subset  \(\mathbf{D}^0 \). We employ this procedure successively until the desired time point \(\tau\) is reached. Let \(\mathbf{D}^{*}\) denote the state space of \(d\) at time point \(\tau\). Here, we must choose \(\varepsilon>0\) such that, \(\mathbf{D}^{*}\) must also be in the feasible region of the  CVX. Now, we can construct the feasible state space for \(c\) denoted by \(\mathbf{C}^{*}\). Since we know initial domain \(\Omega^0=\mathbf{D}^{0} \times \mathbf{C}^{0}  \), we only need to obtain \((d,c)\) pairs for \(d \in \mathbf{D}^{*}  \setminus \mathbf{D}^{0}\). Then, for a given \(\epsilon >0\), we construct the feasible region \(\mathbf{C}^{*}\) for variable \(c\)
	as follows:
	\[ \mathbf{C}^{*}=\mathbf{C}^{0} \cup \bar{\mathbf{C}} \quad \mbox{    with }  \quad \bar{\mathbf{C}}= \displaystyle \bigcup_{d \in \mathbf{D}^{*} \setminus \mathbf{D}^{0} } \mathbf{C}_d,\]
	where 
	\[\mathbf{C}_d=\Big\{c: max(min(\mathbf{C}^{0})- \epsilon \sigma, 0) \leq c \leq max(\mathbf{C}^{0})+ \epsilon \sigma) \: \wedge  y_1(0)-2d+c \geq 0 \:  \wedge  y_2(0)+2d-c \geq 0  \wedge   d \in \mathbf{D}^{*} \setminus \mathbf{D}^{0} \Big\},\]
	where \(\sigma=\sqrt{\mathsf{E}^{J}[\tilde{C}^2 \mid d]}\). Here \(max(\mathbf{C}^{0})\) and  \(min(\mathbf{C}^{0})\) denote the maximum and the minimum values of \(c\) of pairs \((c, max(\mathbf{D}^{0})) \in \Omega^{0}\),respectively. As a result, we have a feasible region \(\Omega^{*}= \mathbf{D}^{*} \times \mathbf{C}^{*} \) for the CVX. Then, we can solve the corresponding convex optimization problems for each conditional counting probability density \(p_{\tau}(c \mid d), d \in \mathbf{D}^{*}\)  using the CVX. Finally, we can compute the approximate solution of \(p_{\tau}(d,c)\). The resulting algorithm is presented in Algorithm \ref{algo1}.
	{\scriptsize
		\begin{algorithm}[!ht]
			\DontPrintSemicolon
			\KwIn{The state vector \(Y\),  the error bound \( \varepsilon\), 
				a discretization time step \(\Delta\), stoichiometric vectors  \(\mu_{1}^{D}\), \(\mu_{2}^{C}\), 
				the initial domain \(\Omega^{0}=\mathbf{D}^{0} \times \mathbf{C}^{0} \),\(\mathbf{D}^{0}, \mathbf{C}^{0} \subset \mathbb{N}_0\), end of the simulation time 
				\( \tau >0\). }
			\KwOut{ The conditional probability at time point \(\tau\), \(p_{\tau}(c \mid d)\). }
			
			Set \(t=0\).\;
			Calculate  \(p_0(d)\), \(\mathsf{E}_0[C\mid d]\), \(\mathsf{E}_0[\tilde{C}^2\mid d]\) on domain \(\Omega^{0}\) by using a Poisson distribution. \;
			Set \(J=\tau/ \Delta\) and define \(t_j=j \Delta\), \(j=0,1,\ldots,J\).\;
			Set \(p^{0}(d) \equiv p_{0}(d)\), \(\mathsf{E}^{0}[C\mid d] \equiv \mathsf{E}_{0}[C\mid d] \), \(\mathsf{E}^{0}[\tilde{C}^2\mid d] \equiv \mathsf{E}_{0}[\tilde{C}^2\mid d] \). \;
			\For{ \(j=1,2,\ldots,J\)}{
				Solve moment equation of the HME by using any  discretization based numerical method and obtain   \(p^j(d)\), \(\mathsf{E}^{j}[C\mid d]\),  
				\(\mathsf{E}^{j}[\tilde{C}^2\mid d]\) for \(d \in \mathbf{D}^j\)\;
				\uIf { \(p^{j}(max(\mathbf{D}^{j})) > \varepsilon\) }
				{ Set \(\mathbf{D}^{j+1}=  \mathbf{D}^{j} \cup \{max(\mathbf{D}^{j})+1\}\)\;
					Set \(p^{j}(d)=p^{j}(d)\) for \(d \in \mathbf{D}^{j} \) \;
					Define   \(p^j( max(\mathbf{D}^{j})+1 ) = \frac{\displaystyle \sum_{d  \in \mathbf{D}}  p^j(d) }{\mid  \mathbf{D}^j \mid +1}, \quad 
					\mathsf{E}^{j}[C \mid  max(\mathbf{D}^{j})+1 ]= \mathsf{E}^{j}[ C \mid  max(\mathbf{D}^{j})]\), \( \mathsf{E}^{j}[\tilde{C}^2 \mid max(\mathbf{D}^{j})+1] = \mathsf{E}^{j}[ \tilde{C}^2  \mid max(\mathbf{D}^{j})],\)
				}
			}
			Set \(\mathbf{D}^{*}= \mathbf{D}^{J} \) \;
			Set \(\mathbf{C}^{*}= \mathbf{C}^{0} \) \;
			\For {\(d \in \mathbf{D}^{*} \setminus \mathbf{D}^{0}\)} {
				Obtain \[\mathbf{C}_d=\Big\{c: max(min(\mathbf{C}^{0})- \epsilon \sigma, 0) \leq c \leq max(\mathbf{C}^{0})+ \epsilon \sigma)  \wedge  y_1(0)-2d+c \geq 0 \:  \wedge  y_2(0)+2d-c \geq 0  \wedge   d \in \mathbf{D}^{*} \setminus \mathbf{D}^{0} \Big\}\] \;
				\(\mathbf{C}^{*}=\mathbf{C}^{*} \cup \mathbf{C}_d\)
			}
			For each \(d \in \mathbf{D}^{*}\) obtain \(p_{\tau}(c \mid d))\) using the CVX.
			
			\caption{ {Constructing feasible region for the CVX} }
			\label{algo1}
	\end{algorithm}}
	
	In our numerical simulation study, the state of the system is initialized \(y(0)=(50,0)^{T}\) and the reaction rate constants of \(R_1\), \(R_2\) are given by \( \kappa_1=0.2 \mathrm{s}^{-1}\),  
	\( \kappa_2=0.4 \mathrm{s}^{-1}\), respectively. We define 
	\[\Omega= \{ (d,c) : 50-2d+c \geq 0, \: 2d-c \geq 0, \: d,c \in\{0,1,2,\ldots,30\} \}.\]
	The initial state space \(\Omega^{0}\) is 
	\[\Omega^{0}=\{(d,c) \in \Omega: \: (d,c)\in\{0,1,2,\ldots,8\} \}.\]
	In Figures (\ref{fig:state_space1}) and  (\ref{fig:state_space2}), one can see the state space \(\Omega\) shown by the points with only green markers and \(\Omega^{0}\) shown by the points with black edged markers. The threshold for extending the region of the variable \(d\) is \(\varepsilon= 10^{-6}\), and \(\epsilon=2\). We obtain joint counting probability density function at time points \(\tau=0.5\) and \(\tau=1\). We have used the Euler method with fixed time step \(\Delta=10^{-4}\). Figures (\ref{fig:state_space1}) and (\ref{fig:state_space2}) also show the \(\Omega^{*}\) at time points \(\tau=0.5\) and \(\tau=1\), respectively. In both figures, 
	the state space \(\Omega^{*}\) is the union of  the points denoted by markers with black and red edges. Figures (\ref{fig:joint_cme1}) and (\ref{fig:joint_cme2}) 
	show the joint counting probability satisfying the CME given in  Equation (\ref{birth_death_cme}) at time points 
	\(\tau=0.5\) and \(\tau=1\), respectively. Figures (\ref{fig:joint_hme1}) and (\ref{fig:joint_hme2}) indicate the approximate solution of the \(p_{\tau}(d,c)\) satisfying  Equation (\ref{birth_death_hybrid_master}) obtained with Algorithm \ref{algo1} at time points 
	\(\tau=0.5\) and \(\tau=1\), respectively.

	\begin{figure} [htb]
		\centering
		\begin{subfigure}[b]{0.3\textwidth}
			\includegraphics[width=\textwidth]{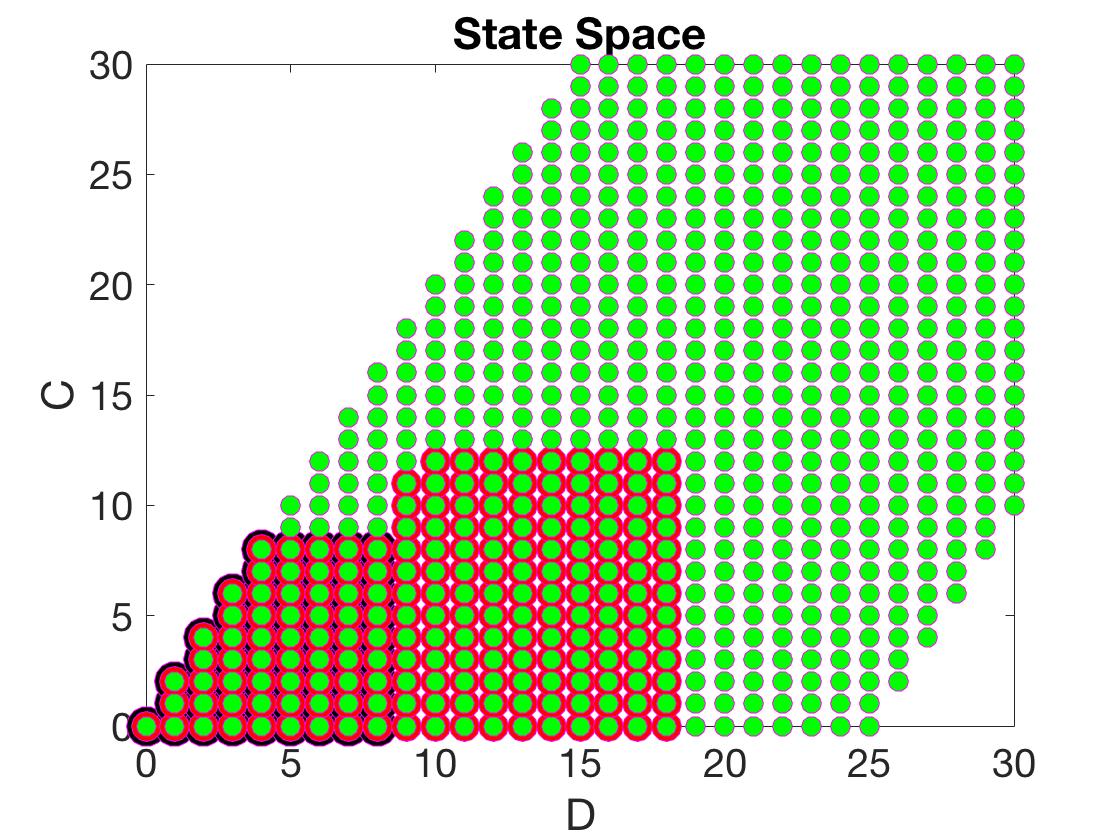}  
			\caption{}
			\label{fig:state_space1}
		\end{subfigure}
		\begin{subfigure}[b]{0.3\textwidth}
			\includegraphics[width=\textwidth]{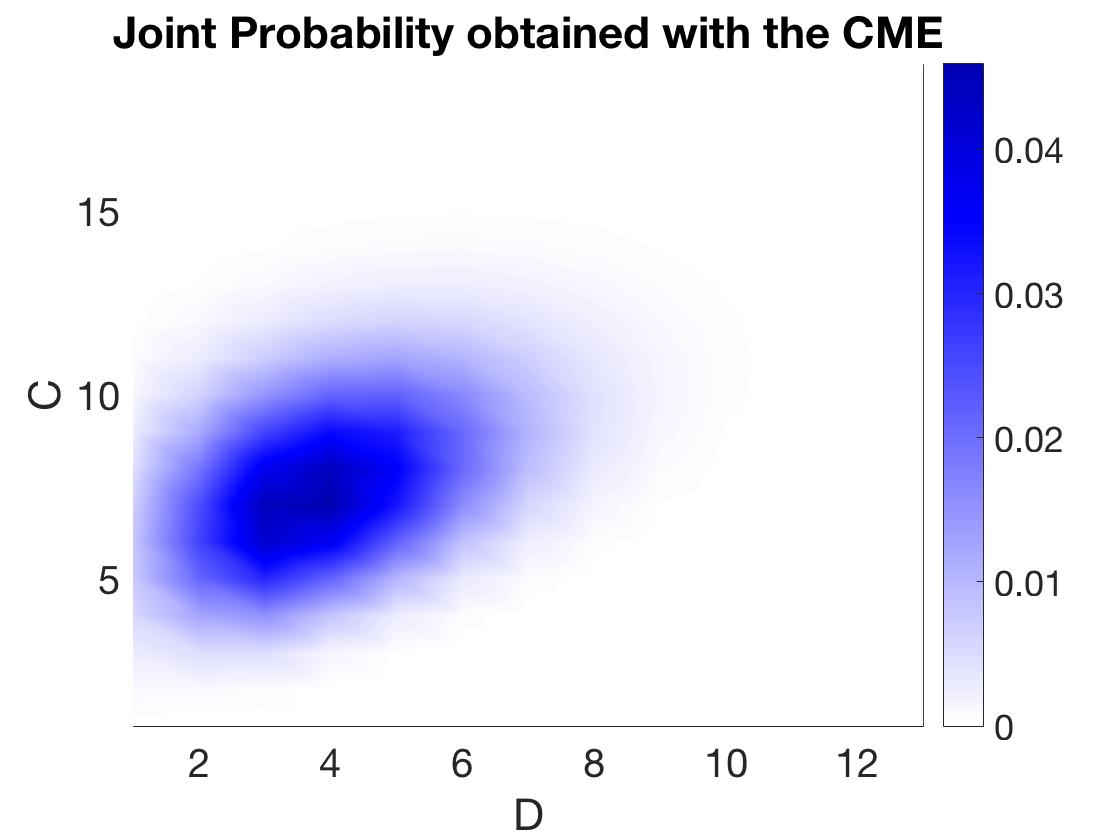}
			\caption{}
			\label{fig:joint_cme1}
		\end{subfigure}
		\begin{subfigure}[b]{0.3\textwidth}
			\includegraphics[width=\textwidth]{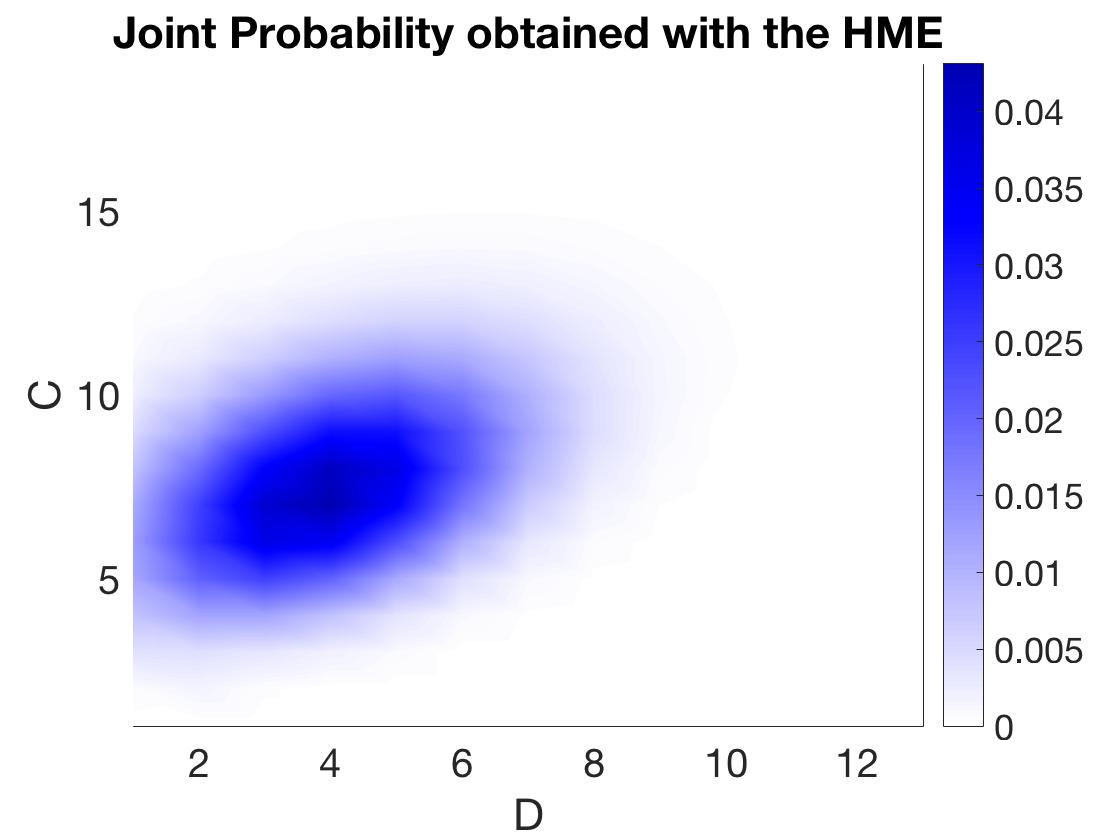}
			\caption{}
			\label{fig:joint_hme1}
		\end{subfigure}
		\caption{The joint counting probability density function at time point \(\tau=0.5\). (a) The state space \(\Omega\) is shown by the points only with green markers; \(\Omega^{0}\) is shown by the points with black edges; and \(\Omega^{*}\) is the union of the points with black and  red edges. (b) The joint counting probability density function satisfying the CME given in  Equation (\ref{birth_death_cme}) (c)The joint counting probability density function satisfying the HME given in  Equation (\ref{birth_death_hybrid_master})  }\label{fig:total1}
	\end{figure}
	
	\begin{figure} [htb]
		\centering
		\begin{subfigure}[b]{0.3\textwidth}
			\includegraphics[width=\textwidth]{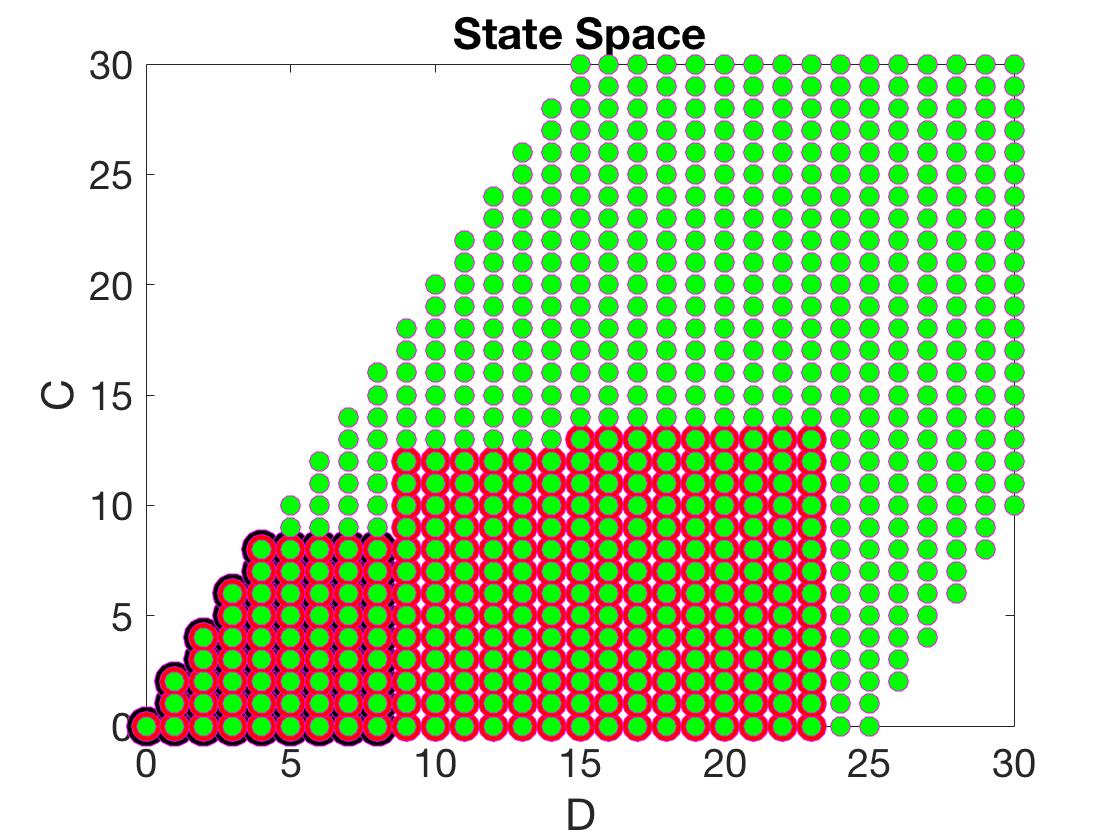}  
			\caption{}
			\label{fig:state_space2}
		\end{subfigure}
		\begin{subfigure}[b]{0.3\textwidth}
			\includegraphics[width=\textwidth]{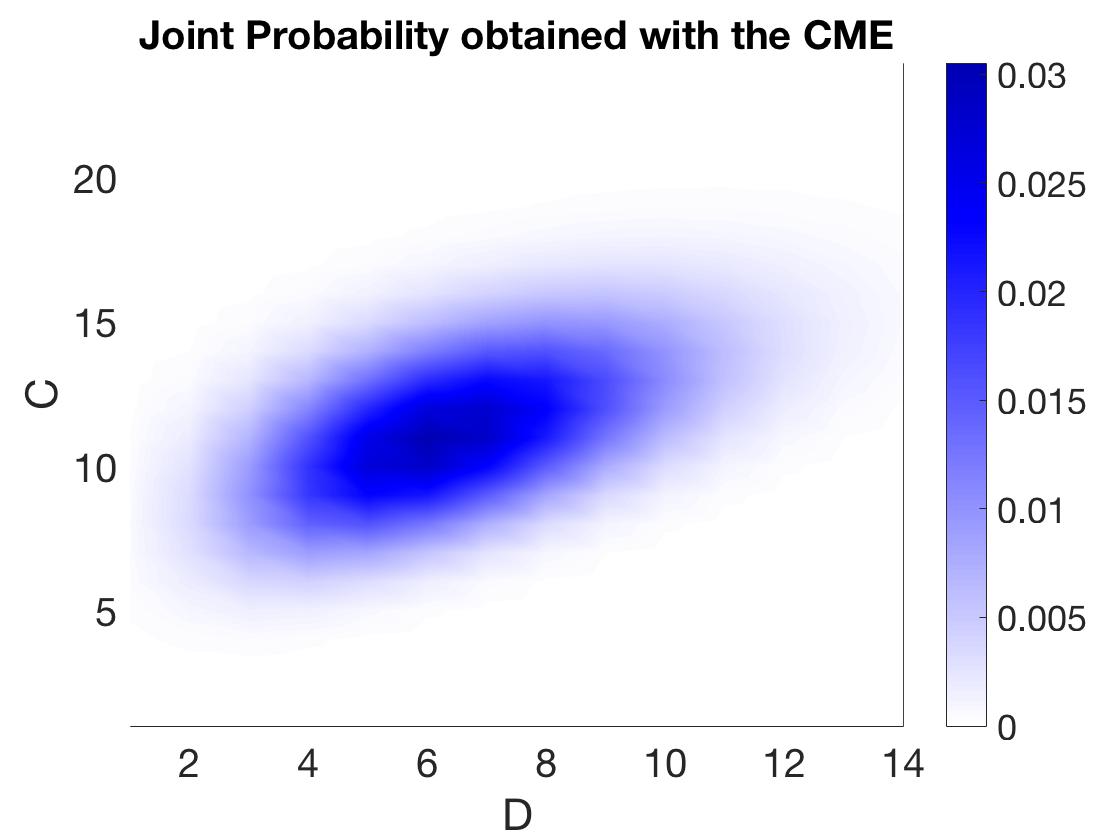}
			\caption{}
			\label{fig:joint_cme2}
		\end{subfigure}
		\begin{subfigure}[b]{0.3\textwidth}
			\includegraphics[width=\textwidth]{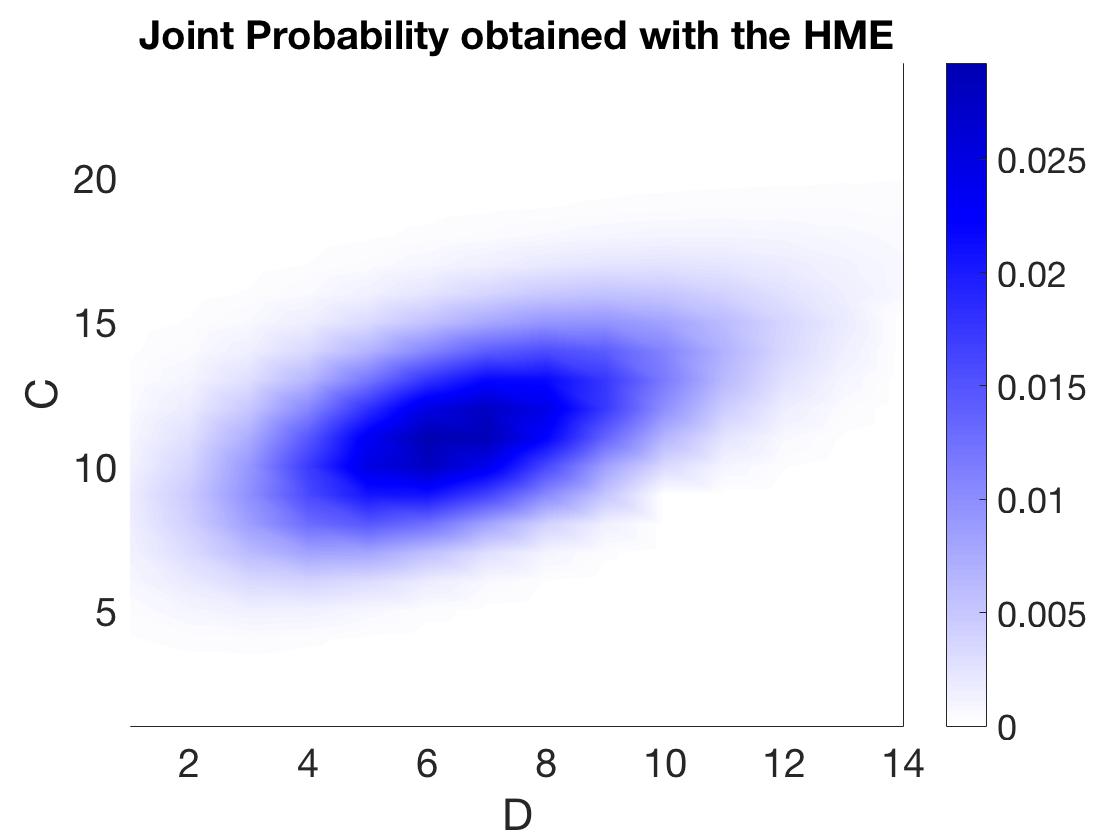}
			\caption{}
			\label{fig:joint_hme2}
		\end{subfigure}
		\caption{The joint counting probability density function at time point \(\tau=1\). (a) The state space \(\Omega\) is shown by the points only with green markers; \(\Omega^{0}\) is shown by the points with black edges; and \(\Omega^{*}\) is the union of the points with black and  red edges. (b) The joint counting probability density function satisfying the CME given in  Equation (\ref{birth_death_cme}) (c)The joint counting probability density function satisfying the HME given in  Equation (\ref{birth_death_hybrid_master})   }\label{fig:total2}
	\end{figure}

	\section{Conclusion}
	\label{conclusion0}
	In this study, we present the hybrid master equation for jump-diffusion approximation, which models systems with multi-scale nature. The idea of jump diffusion approximation is to separate reactions into fast and slow groups based on an obtained error bound. Fast reactions are modeled using diffusion approximation, while Markov chain representation is employed for slow reactions. As a result, the state vector of the system is defined as the summation of the random time  change model and the solution of the Langevin equation. In this study, based on the study of  Pawula \cite{paw:67}, we prove that joint probability density of this hybrid model over reaction counting process satisfies the hybrid master equation, which is the summation of the corresponding chemical master equation and the Fokker-Planck equation. It can be said that  while \cite{gak:015} presents a state vector representation for reaction networks  with multi-scale nature, the current study complements it by obtaining evolution equation for  the corresponding joint probability density over reaction counting process. 
	To solve this equation, we use the same strategy with \cite{hwkt:13}. We write the joint probability density function  as the product of the conditional counting probability density  of the fast reactions conditioned on the counting process of the slow reactions and the marginal probability of the counting process of the slow reactions. To construct the conditional probability density functions at a specific time point, we used the maximum entropy approach. We use the CVX toolbox of the MATLAB to solve the constrained optimization problems. Based on restrictions of the CVX on the dimensionality of the optimization problems, we present a method which constructs feasible regions for the CVX. We apply the method to a gene model.

	\appendix
	\section{Appendix}
	\subsection{ Proof of Lemma \ref{lemma1}}
	\label{proof_lemma1}
	\begin{proof}
		Using Leibniz integral rule  and the boundary conditions gives
		\begin{eqnarray*}
			\frac{\partial}{\partial t} (\mathsf{E}_t[F_t(C)\mid d ] p_{t}(d))&=& \frac{\partial}{\partial t} \Big(\displaystyle \int_{ \mathbb{R}_{\geq 0}^{R-L}}   F_t(c)p_{t}(c\mid d) p_{t}(d) dc \Big )= \displaystyle \int_{ \mathbb{R}_{\geq 0}^{R-L}} F_t(c) \frac{\partial}{\partial t} p_{t}(d,c) \, dc +  \displaystyle \int_{ \mathbb{R}_{\geq 0}^{R-L}} p_{t}(d,c) \frac{\partial}{\partial t} F_t(c)  \, dc 
		\end{eqnarray*}
		Inserting Equation (\ref{prop}) into the first integral yields
		\begin{eqnarray*}
			\frac{\partial}{\partial t}(E_t[ F_t(C)\mid d] p_{t}(d))&=& \displaystyle \int_{ \mathbb{R}_{\geq 0}^{R-L}} F_t(c) \displaystyle \sum_{i \in \mathcal{D}} \kappa_{i}   \displaystyle \prod_{s=1}^{M} \displaystyle \sum_{n=0}^{r_{si}}   \Big(\beta_{s}^{n}(d-\bar{e}_i) \gamma_{s}^{r_{si}-n}(c) p_{t}(d-\bar{e}_i,c)- \beta_{s}^{n}(d) \gamma_{s}^{r_{si}-n}(c)  p_{t}(d,c) \Big) dc \\
			&-&\int_{ \mathbb{R}_{\geq 0}^{R-L}}  F_t(c) \displaystyle \sum_{j \in \mathcal{C}} \displaystyle \frac{\partial}{\partial c_j} \Big(  \kappa_j \displaystyle \prod_{s=1}^{M} \displaystyle \sum_{n=0}^{r_{sj}} \beta_s^{n}(d) \gamma_{s}^{r_{sj}-n}(c)  p_t(d,c)\Big) \, dc \\
			&+& \frac{1}{2} \int_{ \mathbb{R}_{\geq 0}^{R-L}}  F_t(c) \displaystyle \sum_{j  \in  \mathcal{C}} \displaystyle \frac{ \partial^2}{\partial c_j^2}\Big( \kappa_j     \displaystyle \prod_{s=1}^{M} \displaystyle \sum_{n=0}^{r_{sj}} \beta_s^{n}(d) \gamma_{s}^{r_{sj}-n}(c)  p_t(d,c)   \Big) \, dc \\
			&+&\displaystyle \int_{ \mathbb{R}_{\geq 0}^{R-L}}   p_{t}(d,c)  \frac{\partial}{\partial t}  F_t(c) dc.
		\end{eqnarray*}
		Since \(F_t(c)\) is a polynomial function and sufficiently many moments of \(p_{t}(d,c)\) with respect to \(c\) exist, we can manipulate the integral as follows: 
		
		\begin{eqnarray*}
			\frac{\partial}{\partial t}(E_t[ F_t(C)\mid d]  p_{t}(d))&=& \displaystyle \sum_{i \in \mathcal{D}} \kappa_{i}   \displaystyle\int_{ \mathbb{R}_{\geq 0}^{R-L}}  \displaystyle \prod_{s=1}^{M} \displaystyle \sum_{n=0}^{r_{si}} \Big(\beta_{s}^{n}(d-\bar{e}_i)  F_t(c) \gamma_{s}^{r_{si}-n}(c) p_{t}(d-\bar{e}_i,c)- \beta_{s}^{n}(d) F_t(c) \gamma_{s}^{r_{si}-n}(c) p_{t}(d,c) \Big) dc \\
			&-&  \displaystyle \sum_{j \in \mathcal{C}} \kappa_j \displaystyle\int_{ \mathbb{R}_{\geq 0}^{R-L}}  F_t(c) \displaystyle \frac{\partial}{\partial c_i} \Big( \displaystyle \prod_{s=1}^{M} \displaystyle \sum_{n=0}^{r_{sj}} \beta_s^{n}(d) \gamma_{s}^{r_{sj}-n}(c)  p_t(d,c)\Big)dc \\
			&+& \frac{1}{2} \displaystyle \sum_{j \in \mathcal{C}}  \kappa_j \displaystyle \int_{ \mathbb{R}_{\geq 0}^{R-L}}   F_t(c)  \displaystyle \frac{ \partial^2}{\partial c_j^2}\Big(   \displaystyle \prod_{s=1}^{M} \displaystyle \sum_{n=0}^{r_{sj}} \beta_{s}^{n}(d) \gamma_{s}^{r_{sj}-n}(c) p_t(d,c)   \Big) \, dc \\
			&+&\displaystyle \int_{ \mathbb{R}_{\geq 0}^{R-L}}   p_{t}(d,c)  \frac{\partial}{\partial t}  F_t(c) \, dc.
		\end{eqnarray*}
		Here, we want to draw the attention of the reader to the following mean which is used in our equations
		\[\mathsf{E}_t[F_t(C) \gamma_s^{r_{sj}-n}(C)\beta_s(d)\mid d]= \displaystyle  \int_{ \mathbb{R}_{\geq 0}^{R-L}}  F_t(c) \gamma_s^{r_{sj}-n}(c) \beta_s^{n}(d)p_t(c\mid d) \, dc=
		\beta_s(d) \mathsf{E}_t[F_t(C) \gamma_s^{r_{sj}-n}(C) \beta_s^{n}(d)\mid d].\]
		Using this equality and the properties of the FPE will give us the following equation   \cite{csos:16,ov:11}:
		\begin{eqnarray*}
			\frac{\partial}{\partial t}(\mathsf{E}_t[ F_t (C)\mid d]p_t(d))&=&  \displaystyle \sum_{i \in \mathcal{D}}\kappa_{i}  \displaystyle \prod_{s=1}^{M} \displaystyle \sum_{n=0}^{r_{si}} \mathsf{E}_t[F_t (C)\gamma_{s}^{r_{si}-n}(C)\mid d-\bar{e}_i] \beta_{s}^{n}(d-\bar{e}_i) p_{t}(d-\bar{e}_i)- \mathsf{E}_t[F_t(C)\gamma_{s}^{r_{si}-n}(C)\mid d]  \beta_{s}^{n}(d)  p_{t}(d) \\
			&+& \displaystyle \sum_{j \in \mathcal{C}} \kappa_j  \displaystyle \prod_{s=1}^{M} \displaystyle \sum_{n=0}^{r_{sj}} \beta_{s}^{n}(d) \mathsf{E}_t[\gamma_{s}^{r_{sj}-n}(C) \frac{\partial}{\partial c_j} F_t(C)\mid d]p_t(d)\\
			&+& \frac{1}{2}  \displaystyle \sum_{ j \in \mathcal{C}} \kappa_j   \displaystyle \prod_{s=1}^{M} \displaystyle \sum_{n=0}^{r_{sj}} \beta_{s}^{n}(d) \mathsf{E}_t[\gamma_{s}^{r_{sj}-n}(C) \displaystyle \frac{ \partial^2}{\partial c_j^2 } F_t(C) \mid d ]  p_t(d) \\
			&+& \mathsf{E}_t[ \displaystyle \frac{\partial}{\partial t} F_t(C)\mid d] p_t(d).
		\end{eqnarray*}
		
		Define
		\begin{eqnarray*}
			\tilde{\Gamma}(\beta(d), F_t(c)\gamma(c))&=& \displaystyle \sum_{i \in \mathcal{D}} ( \mathcal{F}^{-\bar{e}_i} -I) \Big( \kappa_i \displaystyle \prod_{s=1}^{M}   \displaystyle \sum_{n=0}^{r_{si}} \beta_{s}^{n}(d) \mathsf{E}_{t}[ F_t(C) \gamma_{s}^{r_{sj}-n}(C) \mid d] \Big)\\
			&+& \displaystyle \sum_{j \in \mathcal{C}}  \kappa_{j} \displaystyle \prod_{s=1}^{M} \displaystyle \sum_{n=0}^{r_{sj}}
			\beta_{s}^{n}(d) \mathsf{E}_{t}[ \gamma_{s}^{r_{sj}-n}(C) \frac{\partial}{\partial c_j} F_t(C) \mid d]+\frac{1}{2} 
			\displaystyle \sum_{j \in \mathcal{C}} \kappa_{j} \prod_{s=1}^{M}   \displaystyle \sum_{n=0}^{r_{sj}} \mathsf{E}_t[ 
			\gamma_{s}^{r_{sj}-{n}}(C)  \frac{\partial^{2}}{\partial c_j^{2}} F_t(C) \mid d]
		\end{eqnarray*}

		Then, we obtain 
		\[ \frac{\partial}{\partial t}( \mathsf{E}_t[F_t(C) \mid d] p_t(d))=\tilde{\Gamma}(\beta(d), F_t(c)\gamma(c))p_t(d)+ \mathsf{E}_t[ \displaystyle \frac{\partial}{\partial t} F_t(C)\mid d] p_t(d) \]
		
		which completes the proof. 
	\end{proof}

	\subsection{ Mean of the propensity functions }
	\label{cons_mean} 
	To express \(\mathsf{E}_t[\gamma_s^{r_{si}-n}(C)\mid d]\)  in terms of \(\mathsf{E}_t[C_m\mid d]\), \(\mathsf{E}_t[\tilde{C}^M\mid d]\), we will use the Taylor series expansion of 
	\(\gamma_s^{r_{si}-n}(C)\) around \(\mathsf{E}_t[C\mid d]\). The Taylor polynomial of degree \(q\) for  
	\(\gamma_s^{r_{si}-n}(C)\) around \(\mathsf{E}_t[C\mid d]\) has the following form
	\begin{equation}
	\label{taylor}
	\mathrm{G}_q(c)= \displaystyle \sum_{k_1+\ldots+k_{R-L} \leq q}
	\frac{\partial_{1}^{k_1} \ldots  \partial_{R-L}^{k_{R-L}} } {k_1!  \ldots k_{R-L}!} \gamma_{s}^{r_{si}-n} (\mathsf{E}_t[C\mid d ]) (c_{k_1}-\mathsf{E}_t[ C_{k_1}\mid d])^{k_1} \ldots (c_{k_{R-L}} -\mathsf{E}_t[C_{k_{R-L}} \mid d])^{k_{R-L}},
	\end{equation}
	where \(\partial_j^{\ell}= \displaystyle \frac{\partial ^{\ell}}{\partial c_j^{\ell}}, \: \ell=k_1,\ldots,k_{R-L}\). In general cellular reactions are unimolecular or bimolecular. Therefore, the third and the higher order derivatives will be zero, meaning that the conditional mean of the  function \(\gamma_{s}^{r_{si}-n}(C)\) satisfies \cite{hwkt:13} 
	\begin{equation}
	\label{taylor_prop}
	\mathsf{E}_t[\gamma_{s}^{r_{si}-n}(C)\mid d ] = \gamma_s^{r_{si}-n}(\mathsf{E}_t[ C\mid d])+\frac{1}{2}\displaystyle \sum_{k,\ell=1}^{R-L} \frac{\partial^2}{\partial c_k   \partial c_{\ell}} \gamma_{s}^{r_{si}-n}(\mathsf{E}_t[C\mid d]) \mathsf{E}_t[ \tilde{C}^{e_k+e_\ell} \mid d].
	\end{equation}
	Here, we use the fact that  \(\mathsf{E}_t[ C_i -\mathsf{E}_t[ C_i\mid d] \mid d]=0\), \(i \in \mathcal{C}\). As a result, if we have  a reaction with linear propensity, the second term in Equation (\ref{taylor_prop}) will also be zero and we will get  \(\mathsf{E}_t[\gamma_s^{r_{si}-n}(C)\mid d ]=\gamma_{s}^{r_{si}-n}(\mathsf{E}_t[ C\mid d])\).

	\subsection{ Proof of Proposition  \ref{prop_moment}}
	\label{proof_prop_moment}
	\begin{proof}
		We will  use the product rule for derivatives as follows: 
		\begin{equation}
		\label{product}
		p_t(d) \frac{\partial}{\partial t} \mathsf{E}_t [C_m\mid d]= \frac{\partial}{\partial t} \Big( \mathsf{E}_t[C_m\mid d] p_t(d)\Big) - \mathsf{E}_t[C_m\mid d] \frac{\partial}{\partial t} p_{t}(d).
		\end{equation}
		By setting \(F_t(c)=c_m\) in Lemma  \ref{lemma1}, we can obtain the first derivative on the right hand-side of Equation (\ref{product}) as follows:
		\[\frac{\partial}{\partial t } \Big(\mathsf{E}_t[C_m\mid d]  p_t(d)\Big)= \tilde{\Gamma}(\beta(d), c_m\gamma(c)) p_t(d)+\mathsf{E}_t[\frac{\partial}{\partial t}C_m\mid d] p_t(d).\]
		By using equalities \( \displaystyle \frac{ \partial}{\partial t} C_m=0\), 
		\(\displaystyle \frac{ \partial^2}{\partial c_j^2} C_m=0\), we get 
		\begin{eqnarray*}
			p_t(d) \frac{\partial}{\partial t } \mathsf{E}_t[C_m\mid d]  &=&\displaystyle   \sum_{i \in \mathcal{D}}  
			( \mathcal{F}^{-\bar{e}_i} -I) \Big( \kappa_i \displaystyle \prod_{s=1}^{M}   \displaystyle \sum_{n=0}^{r_{si}} \beta_{s}^{n}(d) \mathsf{E}_t[C_m \gamma_{s}^{r_{si}-n}(C)\mid d] p_t(d) \Big) \\
			&+& \displaystyle\sum_{j \in \mathcal{C}}  \kappa_j \displaystyle \prod_{s=1}^{M} \displaystyle \sum_{n=0}^{r_{sj}} \beta_s^{n}(d) \mathsf{E}_t[\gamma_{s}^{r_{sj}-n}(C) \delta_{jm}\mid d]p_t(d)- \mathsf{E}_t[C_m\mid d] \frac{\partial}{\partial t }  p_t(d),
		\end{eqnarray*}
		where  \(\delta_{jm}\) is the Kronecker delta function. 
	\end{proof}
	
	In  Equation (\ref{moment}), we have the conditional mean \(\mathsf{E}_t[ \gamma_s^{r_{sj}-n}(C) \delta_{jm}\mid d ]\). If \(j \neq m\), then this term  equals to zero; otherwise, we will have \(\mathsf{E}_t[\gamma_s^{r_{sj}-n}(C)\mid d]\). By using  Equation (\ref{taylor_prop}), we can express  \(\mathsf{E}_t[ \gamma_s^{r_{sj}-n}(C)\mid d ]\) also in terms of the conditional means and the centered conditional moments. Additionally, in  Equation (\ref{moment}), we have \(\mathsf{E}_t[ C_m \gamma_s^{r_{si}-n}(C)\mid d]\), which must also be reformulated in terms of the conditional means  \(\mathsf{E}_t[ C_m\mid d], \, m \in \mathcal{C},\) and the centered conditional moments   \(\mathsf{E}_t[ \tilde{C}^{M}\mid d ]\). To achieve this goal, we will add and subtract \(\mathsf{E}_t[ C_m\mid d] \)
	term to and from \(\mathsf{E}_t[ C_m\gamma_s^{r_{si}-n}(C)\mid d] \) as follows \cite{hwkt:13}:
	\[ \mathsf{E}_t[C_m \gamma_s^{r_{si}-n}(C)\mid d]= \mathsf{E}_t [(C_m-\mathsf{E}_t[ C_m\mid d]+\mathsf{E}_t[ C_m\mid d ]) \gamma_s^{r_{si}-n}(C)\mid d ]=  \mathsf{E}_t[ \tilde{C}^{e_m} \gamma_s^{r_{si}-n}(C)\mid d ]+\mathsf{E}_t[ C_m\mid d ]  \mathsf{E}_t[ \gamma_s^{r_{si}-n}(C)\mid d ] .\] Similarly, adding and subtracting \(E_t[ C_m\mid d-\bar{e}_i]\) terms to and from \(\mathsf{E}_t[C_{m}\gamma_s^{r_{si}-n}(C)\mid d-\bar{e}_i]\) will produce
	\begin{eqnarray*}
		\mathsf{E}_t[ C_m \gamma_s^{r_{si}-n}(C)\mid d-\bar{e}_i ] &=& \mathsf{E}_t[ (C_m-\mathsf{E}_t[ C_m\mid d-\bar{e}_i ]+\mathsf{E}_t[ C_m\mid d-\bar{e}_i]) \gamma_s^{r_{si}-n}(C)\mid d-\bar{e}_i ]\\
		&=&  \mathsf{E}_t[\tilde{\Psi}^{e_m} \gamma_s^{r_{si}-n}(C)\mid d-\bar{e}_i ]+\mathsf{E}_t[C_m\mid d-\bar{e}_i ] \mathsf{E}_t[ \gamma_s^{r_{si}-n}(C)\mid d-\bar{e}_i ],  
	\end{eqnarray*}
	where \(\tilde{\Psi}=c- \mathsf{E}_t[C\mid d-\bar{e}_i ]\) and for \(M \in \mathbb{N}^{R-L}\), \(\tilde{\Psi}^{M}= \displaystyle \prod_{j \in \mathcal{C}} \tilde{\Psi}_j^{M_j}\). 
	Inserting the obtained \(\mathsf{E}_t[C_m \gamma_s^{r_{si}-n}(C)\mid d] \hbox{ and }  \mathsf{E}_t[C_m \gamma_s^{r_{si}-n}(C)\mid d-\bar{e}_i] \) values  into Equation  (\ref{moment})  will give us the following equation
	\begin{eqnarray}
	\nonumber p_t(d) \frac{\partial}{\partial t} \mathsf{E}_t[C_m\mid d] &=& \displaystyle \sum_{i \in \mathcal{D}} \displaystyle \prod_{s=1}^{M} \displaystyle \sum_{n=0}^{r_{si}} \kappa_{i} \Big(
	\mathsf{E}_t[ \tilde{\Psi}^{e_m} \gamma_s^{r_{si}-n}(C)\mid d-\bar{e}_i] +\mathsf{E}_t[ C_m\mid d-\bar{e}_i] \mathsf{E}_t[ \gamma_s^{r_{si}-n}(C) \mid d-\bar{e}_i]\Big) \beta_s^{n}(d-\bar{e}_i) p_t(d-\bar{e}_i)  \\
	&-&  \displaystyle \sum_{i \in \mathcal{D}}  \displaystyle \prod_{s=1}^{M} \displaystyle \sum_{n=0}^{r_{si}} \kappa_{i} \Big( \mathsf{E}_t [\tilde{C}^{e_m} \gamma_s^{r_{si}-n}(C)\mid d] +\mathsf{E}_t[C_m\mid d] \mathsf{E}_t[ \gamma_s^{r_{si}-n}(C) \mid d]\Big) \beta_s^{n}(d) p_t(d) \label{moment3}\\
	\nonumber &+& \displaystyle \sum_{j \in \mathcal{C}}   \kappa_{j} \displaystyle \prod_{s=1}^{M} \displaystyle \sum_{n=0}^{r_{sj}} \Big( \mathsf{E}_t[\gamma_s^{r_{sj}-n}(C) \delta_{jm} \mid d] \beta_s^{n}(d) p_t(d)\Big)- \mathsf{E}_t[C_m \mid d] \frac{\partial}{\partial t} p_t(d).  
	\end{eqnarray}
	
	As mentioned before, the Taylor series expansion of \(\gamma_s^{r_{si}-n}(C)\) around \(\mathsf{E}_t[C\mid d]\) will give us the possibility to reformulate \(\mathsf{E}_t[\gamma_s^{r_{si}-n}(C)\mid d ]\) using the conditional means and the centered conditional moments. Here, we must reformulate \(\mathsf{E}_t[\tilde{C}^{e_m} \gamma_s^{r_{si}-n}(C)\mid d ]\) and 
	\(\mathsf{E}_t[\tilde{\Psi}^{e_m} \gamma_s^{r_{si}-n}(C)\mid d-\bar{e}_i ]\) using the corresponding conditional means and the centered moments. To do this, we will use the Taylor expansion given in Equation (\ref{taylor}) as follows:
	\begin{eqnarray}
	\nonumber \mathsf{E}_t[ \tilde{C}^{M} \gamma_{s}^{r_{si}-n}(C)\mid d ]&=& \gamma_{s}^{r_{si}-n}(\mathsf{E}_t[C\mid d]) \mathsf{E}_t[ \tilde{C}^{M}\mid d] +\displaystyle \sum_{k \in \mathcal{C}} \frac{\partial}{\partial c_k} \gamma_{s}^{r_{si}-n}(\mathsf{E}_t[ C\mid d ])  \mathsf{E}_t[ \tilde{C}^{M+e_k}\mid d]\\
	&+& \frac{1}{2} \displaystyle \sum_{k \in \mathcal{C}} \frac{\partial^2}{\partial c_k  \partial c_{\ell}}  
	\gamma_{s}^{r_{si}-n}(\mathsf{E}_t[ C\mid d])   \mathsf{E}_t[ \tilde{C}^{M+e_k+e_{\ell}}\mid d], \label{taylor21n}
	\end{eqnarray}
	where \(M=(M_1,M_2,\ldots,M_{R-L})^{T} \in \mathbb{N}^{R-L}\). 
	It is clear that \(\mathsf{E}_t[\tilde{\Psi}^{M} \gamma_s^{r_{si}-n}(C)\mid d-\bar{e}_i ]\) can also be reformulated by using the Taylor expansion of \(\gamma_s^{r_{si}-n}(C)\)  around  \(\mathsf{E}_t[C\mid d-\bar{e}_i ]\). Substitution of the new representations of 
	\(\mathsf{E}_t[\tilde{C}^{M} \gamma_s^{r_{si}-n}(C)\mid d]\) and \(\mathsf{E}_t[\tilde{\Psi}^{M} \gamma_s^{r_{si}-n}(C)\mid d-\bar{e}_i ]\), which  only depend on the conditional means and the centered conditional moments conditioned on the corresponding discrete variable into  Equation  (\ref{moment}) will produce \(p_t(d) \displaystyle \frac{\partial}{\partial t} \mathsf{E}_t[C_m\mid d] \) in terms of the marginal probabilities, the conditional means and the centered conditional moments. 
	\subsection{ Proof of Proposition  \ref{prop_var}}
	\label{proof_prop_var}
	\begin{proof}
		Similar to our previous proofs, again we will use the product rule for derivatives as follows: 
		\[p_{t}(d) \frac{\partial}{\partial t } \mathsf{E}_t[\tilde{C}^{M}\mid d ] = \frac{\partial}{\partial t} \Big(\mathsf{E}_t[\tilde{C}^{M}\mid d] p_{t}(d)\Big)- \mathsf{E}_t[\tilde{C}^{M}\mid d]\frac{\partial}{\partial  t} p_{t}(d).\]
		The first term in the right hand-side of the  equation above can be obtained from Lemma \ref{lemma1}  choosing  \(F(c)=\tilde{c}^{M}\). Then, we obtain
		\begin{equation*}
		\frac{\partial}{\partial t } \Big(\mathsf{E}_t[\tilde{C}^{M}\mid d] p_t(d)\Big)=  \tilde{ \Gamma}(\beta(d), \tilde{c}^{M} \gamma(c))p_t(d)+ \mathsf{E}_t[\frac{\partial}{\partial t} \tilde{C}^{M}\mid d] p_t(d).
		\end{equation*}
		Since, we have  
		\begin{eqnarray*}
			\displaystyle \frac{\partial}{\partial c_i} \tilde{C}^M&=& \displaystyle \frac{\partial}{\partial c_i} \displaystyle \prod_{k \in \mathcal{C}} (c_k- \mathsf{E}_t[C_k\mid d])^{M_{k}}=M_i\tilde{C}^{M-e_i} \\
			\displaystyle \frac{\partial^2}{\partial c_j^2 } \tilde{C}^M&=& \displaystyle \frac{\partial^2}{\partial c_j^2} \displaystyle \prod_{k \in \mathcal{C}} (c_k- \mathsf{E}_t[C_k])^{M_{k}}=M_j (M_j-1)\tilde{C}^{M-2e_j} \\
			\mathsf{E}_t[\frac{\partial}{\partial t} \tilde{C}^{M}\mid d] &=& -\displaystyle \sum_{j \in \mathcal{C}} M_j \mathsf{E}_t[\tilde{C}^{M-e_j}\mid d] p_{t}(d) \frac{\partial}{\partial t} \mathsf{E}_t [C_j\mid d].
		\end{eqnarray*}
		We get 
		\begin{eqnarray*}
			p_t(d)\displaystyle \frac{\partial}{\partial t} \mathsf{E}_t[\tilde{C}^{M}\mid d]&=&  \displaystyle \sum_{i \in \mathcal{D}}
			\kappa_{i} \displaystyle \prod_{s=1}^{M} \displaystyle \sum_{n=0}^{r_{si}}  \Big( \beta_{s}^{n}(d-\bar{e}_i) \mathsf{E}_t[\tilde{C}^{M} \gamma_{s}^{r_{si}-n}(C)\mid d-\bar{e}_i] p_{t}(d-\bar{e}_i)- \beta_{s}^{n}(d) \mathsf{E}_t[\tilde{C}^{M} \gamma_{s}^{r_{si}-n}(C)\mid d]  p_{t}(d) \Big)\\
			&+& \displaystyle  \sum_{j \in \mathcal{C}} \kappa_j  \displaystyle \prod_{s=1}^{M} \displaystyle \sum_{n=0}^{r_{sj}} \beta_{s}^{n}(d) \mathsf{E}_t[M_j\gamma_{s}^{r_{sj}-n}(C)  \tilde{C}^{M-e_j}\mid d]p_t(d)\\
			&+& \frac{1}{2}  \displaystyle \sum_{j \in \mathcal{C}}   \kappa_j   \displaystyle \prod_{s=1}^{M} \displaystyle \sum_{n=0}^{r_{sj}} \beta_s^{n}(d) \mathsf{E}_t[M_j (M_j-1) \gamma_{s}^{r_{sj}-n}(C)  \tilde{C}^{M-2e_j} \mid d ]  p_t(d) \\
			&-& \displaystyle \sum_{j \in \mathcal{C}}  M_j \mathsf{E}_{t}[\tilde{C}^{M-e_j}\mid d] p_t(d) \frac{\partial}{\partial t }\mathsf{E}_t[C_j\mid d]- \mathsf{E}_t[\tilde{C}^{M}\mid d] \frac{\partial}{\partial t} p_t(d) .
		\end{eqnarray*}
		As a result, we obtain
		\begin{eqnarray*} 
			p_t(d) \displaystyle \frac{\partial}{\partial t} \mathsf{E}_t[\tilde{C}^{M}\mid d]&=&  \displaystyle   \sum_{i \in \mathcal{D}}  
			( \mathcal{F}^{-\bar{e}_i} -I) \Big( \kappa_i \displaystyle \prod_{s=1}^{M}   \displaystyle \sum_{n=0}^{r_{si}} \beta_{s}^{n}(d) \mathsf{E}_t[\tilde{C}^{M} \gamma_{s}^{r_{si}-n}(C)\mid d] p_t(d) \Big) \\
			&+& \displaystyle  \sum_{j \in \mathcal{C}} \kappa_j  \displaystyle \prod_{s=1}^{M} \displaystyle \sum_{n=0}^{r_{sj}} \beta_{s}^{n}(d) \mathsf{E}_t[M_j\gamma_{s}^{r_{sj}-n}(C)  \tilde{C}^{M-e_j}\mid d]p_t(d)\\
			&+& \frac{1}{2}  \displaystyle \sum_{j \in \mathcal{C}}   \kappa_j   \displaystyle \prod_{s=1}^{M} \displaystyle \sum_{n=0}^{r_{sj}} \beta_s^{n}(d) \mathsf{E}_t[M_j (M_j-1) \gamma_{s}^{r_{sj}-n}(C)  \tilde{C}^{M-2e_j} \mid d ]  p_t(d) \\
			&-& \displaystyle \sum_{j \in \mathcal{C}}  M_j \mathsf{E}_{t}[\tilde{C}^{M-e_j}\mid d] p_t(d) \frac{\partial}{\partial t }\mathsf{E}_t[C_j\mid d]- \mathsf{E}_t[\tilde{C}^{M}\mid d] \frac{\partial}{\partial t} p_t(d),
		\end{eqnarray*}
		which completes our proof. 
	\end{proof}
	
	To formulate the right hand-side of  Equation (\ref{final}) in terms of  the marginal probabilities, the conditional means  and the centered conditional moments, we must restate  \(\mathsf{E}_t[ \tilde{C}^{M} \gamma_s^{r_{si}-n}(C)\mid d-e_i ]\), \(\mathsf{E}_t[ \tilde{C}^{M} \gamma_s^{r_{si}-n}(C)\mid d]\) terms using these terms.  \(\mathsf{E}_t[ \tilde{C}^{M} \gamma_s^{r_{si}-n}(C)\mid d]\) , \(\mathsf{E}_t[M_j\gamma_{s}^{r_{sj}-n}(C)  \tilde{C}^{M-e_j}\mid d]\), \( \mathsf{E}_t[M_i M_j \gamma_{s}^{r_{sk}-n}(C)  \tilde{C}^{M-e_i-e_j} \mid d ]\) can be expressed  utilizing the corresponding Taylor series expansion given in  Equation (\ref{taylor21n}). 
	To express \(\mathsf{E}_t[ \tilde{C}^{M} \gamma_s^{r_{si}-n}(C)\mid d-e_i ]\) in terms of the marginal probabilities, the conditional means and the centered conditional moments, we will add and subtract \(\mathsf{E}_t[ C\mid d-e_i ]\) to and from \(\tilde{C}^{M}\) as follows:
	\begin{eqnarray*}
		\tilde{c}^{M}&=&\Big(c- \mathsf{E}_t[ C\mid d]\Big)^{M}=\Big(c- \mathsf{E}_t[C\mid d-\bar{e}_i]+ \mathsf{E}_t[ C\mid d-\bar{e}_i]- \mathsf{E}_t[C\mid d]\Big)^{M} \\
		&=&\displaystyle \sum_{0 \leq k \leq M} {M \choose k} \Big(\mathsf{E}_t[ C\mid d-\bar{e}_i] - \mathsf{E}_t[ C\mid d]\Big)^{M-k} \tilde{\Psi}^{k}.
	\end{eqnarray*}
	Then, we obtain 
	\begin{eqnarray*}
		\mathsf{E}[\tilde{C}^{M} \gamma_s^{r_{si}-n}(C)\mid d-\bar{e}_i]=\displaystyle \sum_{0 \leq k \leq M} {M \choose k} \Big(\mathsf{E}_t[ C\mid d-\bar{e}_i] - \mathsf{E}[ C\mid d]\Big)^{M-k} \mathsf{E}[\tilde{\Psi}^{k} \gamma_{s}^{r_{si}-n}(C)\mid d-\bar{e}_i].
	\end{eqnarray*}
Using the Taylor series expansion of  \(\gamma_{s}^{r_{si}-n}(C)\) around \(\mathsf{E}_{t}[C \mid  d-\bar{e}_i]\) gives us the corresponding Taylor series representation for 
	\(\mathsf{E}_t[\tilde{\Psi}^{k} \gamma_{s}^{r_{si}-n}(C)\mid d-\bar{e}_i]\). As a result, we can obtain the right hand-side of Equation (\ref{final}) in terms of the marginal probabilities, the conditional means and the centered moments.

	\bibliographystyle{plain}
	\bibliography{draft_hybrid_master}

\begin{thebibliography}{10}

\bibitem{ab:10}
R.~V. Abramov.
\newblock The multidimensional maximum entropy moment problem: a review of
  numerical methods.
\newblock {\em Communications in Mathematical Sciences}, 8(2):377-- 392, 2010.

\bibitem{ak:11}
D.~F. Anderson and T.~G. Kurtz.
\newblock Continuous time {M}arkov chain models for chemical reaction networks.
\newblock In H.~Koeppl, Gianluca Setti, Mario~di Bernardo, and Douglas
  Densmore, editors, {\em Design and Analysis of Biomolecular Circuits}.
  Springer-Verlag, 2011.

\bibitem{amw:15}
A.~Andreychenko, L.~Mikeev, and V.~Wolf.
\newblock Model reconstruction for moment-based stochastic chemical kinetics.
\newblock {\em ACM Trans. Model. Comput. Simul.}, 25(2), 2015.

\bibitem{bv:04}
S.~Boyd and L.~Vandenberghe.
\newblock {\em Convex Optimization}.
\newblock Cambridge University Press, 2004.

\bibitem{bre:13}
G.~L. Bretthorst.
\newblock The maximum entropy method of moments and bayesian probability
  theory.
\newblock {\em AIP Conf. Proc.}, 1553:3--15, 2013.

\bibitem{ckl:16}
L.~Cardelli, M.~Kwiatkowska, and L.~Laurenti.
\newblock A stochastic hybrid approximation for chemical kinetics based on the
  linear noise approximation.
\newblock In {\em CMSB}, Lecture Notes in Computer Science, pages 147--167.
  Springer, 2016.

\bibitem{ce:18}
A.~Chevallier and S.~Engblom.
\newblock Pathwise error bounds in multiscale variable splitting methods for
  spatial stochastic kinetics.
\newblock {\em SIAM J. NUMER. ANAL.}, 58(1):469--498, 2018.

\bibitem{cdr:09}
A.~Crudu, A.~Debussche, and O.~Radulescu.
\newblock Hybrid stochastic simplifications for multiscale gene networks.
\newblock {\em BMC Systems Biology}, 3(89), 2009.

\bibitem{csos:16}
B.~Cseke, D.~Schnoerr, M.~Opper, and G.~Sanguinetti.
\newblock Expectation propagation for continuous time stochastic processes.
\newblock {\em Journal of Physics A: Mathematical and Theoretical}, 49(49),
  2016.

\bibitem{dez:16}
A.~Duncan, R.~Erban, and K.~Zygalakis.
\newblock Hybrid framework for the simulation of stochastic chemical kinetics.
\newblock {\em Journal of Computational Physics}, 326, 2016.

\bibitem{ee:10}
A.~Eldar and M.~B. Elowitz.
\newblock Functional roles for noise in genetic circuits.
\newblock {\em Nature}, 467:167--173, 2010.

\bibitem{eng:06}
S.~Engblom.
\newblock Computing the moments of high dimensional solutions of the master
  equation.
\newblock {\em Applied Mathematics and Computation}, 180:498--515, 2006.

\bibitem{ehl:17}
S.~Engblom, A.~Hellander, and P.~L{\"{o}}tstedt.
\newblock {\em Multiscale Simulation of Stochastic Reaction-Diffusion
  Networks}, pages 55--79.
\newblock Stochastic Processes, Multiscale Modeling, and Numerical Methods for
  Computational Cellular Biology. Springer, 2017.

\bibitem{ff:02}
N.~Fedoroff and W.~Fontana.
\newblock Small numbers of big molecules.
\newblock {\em Science}, 297, 2002.

\bibitem{fcs:10}
N.~Friedman, L.~Cai, and X.S. Xie.
\newblock Stochasticity in gene expression as observed by single-molecule
  experiments in live cells.
\newblock {\em Israel Journal of Chemistry}, 49:333--342, 2010.

\bibitem{gak:015}
A.~Ganguly, D.~Alt{\i}ntan, and H.~Koeppl.
\newblock Jump-diffusion approximation of stochastic reaction dynamics: Error
  bounds and algorithms.
\newblock {\em Multiscale Model. Simul.}, 13(4):1390--1419, 2015.

\bibitem{gb:00}
M.A. Gibson and J.~Bruck.
\newblock Efficient exact stochastic simulation of chemical systems with many
  species and many channels.
\newblock {\em The Journal of Physical Chemistry A}, 104(9):1876--1889, 2000.

\bibitem{gill:76}
D.~T. Gillespie.
\newblock A general method for numerically simulating the stochastic time
  evolution of coupled chemical reactions.
\newblock {\em J. Comput. Phys.}, 22:403--434, 1976.

\bibitem{gill:80}
D.~T. Gillespie.
\newblock Approximating the master equation by {F}okker-{P}lanck type equations
  for singlevariable chemical systems.
\newblock {\em The Journal of Chemical Physics}, 72(5363), 1980.

\bibitem{gill:00}
D.~T. Gillespie.
\newblock The chemical {L}angevin equation.
\newblock {\em Journal of Chemical Physics}, 113(1):297--306, 2000.

\bibitem{gill:02}
D.~T. Gillespie.
\newblock The chemical {L}angevin and {F}okker-{P}lanck equations for the
  reversible isomerization reaction.
\newblock {\em J. Phys. Chem. A}, 106:5063--5071, 2002.

\bibitem{gill:07}
D.~T. Gillespie.
\newblock Stochastic simulation of chemical kinetics.
\newblock {\em Annu. Rev. Phys. Chem.}, 58:35--55, 2007.

\bibitem{gill:92}
D.T. Gillespie.
\newblock A rigorous derivation of the chemical master equation.
\newblock {\em Physica A}, 188:404--425, 1992.

\bibitem{cvx}
M.~Grant and S.~Boyd.
\newblock {CVX}: Matlab software for disciplined convex programming, version
  2.1, March 2014.

\bibitem{gts:11}
R.~Grima, P.~Thomas, and A.~V. Straube.
\newblock How accurate are the nonlinear chemical {F}okker-{P}lanck and
  chemical {L}angevin equations?
\newblock {\em Journal of Chemical Physics}, 135(8), 2011.

\bibitem{hwkt:13}
J.~Hasenauer, V.~Wolf, A.~Kazeroonian, and F.~J. Theis.
\newblock Method of conditional moments ({MCM}) for the chemical master
  equation : A unified framework for the method of moments and hybrid
  stochastic-deterministic models.
\newblock {\em J. Math. Biol.}, 69(3):687--735, 2014.

\bibitem{jah:11}
T.~Jahnke.
\newblock On reduced models for the chemical master equation.
\newblock {\em Multiscale Model. Simul.}, 9(4):1646--1676, 2011.

\bibitem{kam:82}
N.~G.~van Kampen.
\newblock The diffusion approximation for {M}arkov process.
\newblock In {\em Thermodynamics and Kinetics of Biological Processes}, pages
  185--195. Walter de Gruyter and Co., 1982.

\bibitem{kur:78}
Thomas~G. Kurtz.
\newblock Strong approximation theorems for density dependent {M}arkov cahins.
\newblock {\em Stochast. Process. Appl.}, 6(3):177--191, 1978.

\bibitem{lkk:09}
C.~H. Lee, K.~H. Kim, and P.~Kim.
\newblock A moment closure method for stochastic reaction networks.
\newblock {\em The Journal of Chemical Physics}, 130:134107, 2009.

\bibitem{ov:11}
D.~L. Otten and P.~Vedula.
\newblock A quadrature based method of moments for nonlinear {F}okker-{P}lanck
  equations.
\newblock {\em Journal of Statistical Mechanics: Theory and Experiment},
  2011(9), 2011.

\bibitem{paw:67}
R.~F. Pawula.
\newblock Generalizations and extensions of the {F}okker-{P}lanck {K}olmogorov
  equations.
\newblock {\em IEEE Transactions on Information Theory}, 13(1), 1967.

\bibitem{rh:89}
H.~Risken and H.~Haken.
\newblock {\em {The {F}okker-{P}lanck Equation: Methods of Solution and
  Applications Second Edition}}.
\newblock Springer, 1989.

\bibitem{sha:48}
C.~E. Shannon.
\newblock A mathematical theory of communication.
\newblock {\em The Bell System Technical Journal}, 27:379–--423, 1948.

\bibitem{sh:11}
A.~Singh and J.~P. Hespanha.
\newblock Approximate moment dynamics for chemically reacting systems.
\newblock {\em IEEE Trans. on Automat. Contr.}, 56(2):414--418, 2011.

\bibitem{wil:06}
D.J. Wilkinson.
\newblock {\em Stochastic modelling for systems biology}.
\newblock Chapman \& Hall/CRC mathematical and computational biology series.
  Boca Raton, FL : Taylor \& Francis, 2006.

\bibitem{wgmh:10}
V.~Wolf, R.~Goel, M.~Mateescu, and T.~A. Henzinger.
\newblock Solving the chemical master equation using sliding windows.
\newblock {\em BMC Systems Biology}, 4, 2010.

\end{thebibliography}
	
\end{document}